\documentclass[final]{siamart0216}
\usepackage{amsfonts}

\usepackage{amsfonts,amsmath,amssymb}
\usepackage{mathrsfs,mathtools,stmaryrd,wasysym}
\usepackage{enumerate}
\usepackage{xspace,mydef} 
\usepackage{esint}
\usepackage{graphicx,psfrag}
\usepackage{hyperref}
\usepackage{pgf,tikz}
\usetikzlibrary{arrows}
\usepackage{rotating}
\DeclareMathAlphabet{\mathpzc}{OT1}{pzc}{m}{it}

\DeclareMathOperator{\supp}{supp}

\newsiamremark{remark}{Remark}

\newcommand{\dist}{ {\textup{\textsf{d}}}_{z} }
\newcommand{\Sides}{\mathscr{S}}
\newcommand{\Ne}{\mathcal{N}}
\newcommand{\E}{\mathscr{E}}

\newcommand{\TheTitle}{A posteriori error estimates for the stationary Navier Stokes equations with Dirac measures}
\newcommand{\ShortTitle}{A posteriori error estimates with Dirac measures}
\newcommand{\TheAuthors}{A. Allendes, E.~Ot\'arola, A. J.~Salgado}

\headers{\ShortTitle}{\TheAuthors}

\title{{\TheTitle}\thanks{AA has been partially supported by CONICYT through FONDECYT project 1170579. EO has been partially supported by CONICYT through FONDECYT project 11180193. AJS has been partially supported by NSF grant DMS-1720213.}}

\author{
  Alejandro Allendes\thanks{Departamento de Matem\'atica, Universidad T\'ecnica Federico Santa Mar\'ia, Valpara\'iso, Chile.
    (\email{alejandro.allendes@usm.cl}, \url{http://aallendes.mat.utfsm.cl/}).}
    \and
  Enrique Ot\'arola\thanks{Departamento de Matem\'atica, Universidad T\'ecnica Federico Santa Mar\'ia, Valpara\'iso, Chile.
    (\email{enrique.otarola@usm.cl}, \url{http://eotarola.mat.utfsm.cl/}).}
  \and
  Abner J.~Salgado\thanks{Department of Mathematics, University of Tennessee, Knoxville, TN 37996, USA.
    (\email{asalgad1@utk.edu}, \url{http://www.math.utk.edu/\string~abnersg})}
}

\ifpdf
\hypersetup{
  pdftitle={\TheTitle},
  pdfauthor={\TheAuthors}
}
\fi

\date{Submitted \today.}

\begin{document}

\maketitle

\begin{abstract}
In two dimensions, we propose and analyze an a posteriori error estimator for finite element approximations of the stationary Navier Stokes equations with singular sources on Lipschitz, but not necessarily convex, polygonal domains. Under a smallness assumption on the continuous and discrete solutions, we prove that the devised error estimator is reliable and locally efficient. We illustrate the theory with numerical examples.
\end{abstract}

\begin{keywords}
A posteriori error estimates, Navier Stokes equations, Dirac measures, Muckenhoupt weights.
\end{keywords}

% REQUIRED
\begin{AMS}
35Q35,         % PDEs in connection with fluid mechanics
35Q30,         % Navier-Stokes equations 
35R06,         % Partial differential equations with measure
76Dxx,         % Incompressible viscous fluids
65N15,         % Error bounds
65N30,         % Finite elements, Rayleigh-Ritz and Galerkin methods, finite methods
65N50.         % Mesh generation and refinement
\end{AMS}

\section{Introduction}
\label{sec:intro}

Let $\Omega \subset \R^2$ be an open and bounded domain with Lipschitz boundary $\partial \Omega$. In this work we will be interested in the design and analysis of a posteriori error estimates for finite element approximations of the stationary Navier Stokes problem
\begin{equation}
\label{eq:NSEStrong}
    -\Delta \ue + (\ue \cdot \GRAD) \ue + \GRAD \pe = \bF \delta_z \text{ in } \Omega, \qquad
    \DIV \ue = 0 \text{ in } \Omega, \qquad
    \ue = \boldsymbol0 \text{ on } \partial\Omega,
\end{equation}
where $\delta_z$ corresponds to the Dirac delta supported at the interior point $z \in \Omega$ and $\bF \in \mathbb{R}^2$. Here, $\bu$ represents the velocity of the fluid, $\pe$ the pressure, and  $\bF \delta_z$ is an externally applied force. Notice that, for simplicity, we have taken the viscosity to be equal to one. 

Since the stationary Navier Stokes equations model the motion of a stationary, incompressible, Newtonian fluid, it is no surprise that their analysis and approximation, at least in energy--type spaces, is very well developed; see, for instance, \cite{MR1846644,MR851383,MR2986590,MR1928881,MR2258988} for an account of this theory.

On the other hand, there are situations where one wishes to allow this model to be driven by singular forces, like in \eqref{eq:NSEStrong}. As a first example of this we mention \cite{Lacouture2015187}, where the linear version of \eqref{eq:NSEStrong} is considered, and it is argued that it can be used to model the movement of active thin structures in a viscous fluid. A numerical scheme is proposed, but no complete analysis of this method is provided. Local error estimates, away from the support of the delta, were later provided in \cite{MR3854357}.

A second example comes from PDE--constrained optimization (optimal control). Reference \cite{MR3936891} sets up a problem where the state is governed by the stationary Navier Stokes equations, but with a forcing (control) that is measure valued, like in \eqref{eq:NSEStrong}. The motivation behind this is what the authors denote \emph{sparsity} of the control, meaning that its support is small, even allowing it to have Lebesgue measure zero. The analysis of \cite{MR3936891} assumes that the domain has $C^2$ boundary, and seeks for a solution to \eqref{eq:NSEStrong} in $\bW^{1,q}_0(\Omega) \times L^q(\Omega)/\R$ with $q \in [4/3,2)$. In this setting a complete existence theory for the state is provided, and the optimization problem is analyzed. Necessary and sufficient optimality conditions are deduced. This work, however, is not concerned with approximation.

In this work we continue our program aimed at developing numerical methods for models of fluids under singular forces. The guiding principle that we follow is that by introducing a weight, and working in the ensuing weighted function spaces, we can allow for data that is singular, so that \eqref{eq:NSEStrong} fits our theory. We immediately must comment that the literature already presents an analysis of the stationary Navier Stokes equations on Muckenhoupt weighted spaces; see \cite{MR2574796}. This paper however, requires the domain to be $C^{1,1}$, which is not suitable for a finite element approximation. We, in contrast, assume only that the domain is Lipschitz. In \cite{OS:17infsup} we developed existence and uniqueness for the Stokes problem over a reduced class of weighted spaces, see Definition~\ref{def:ApOmega} below. The numerical analysis of this linear model is presented in \cite{DuranOtarolaAJS,MR3892359}, where a priori and a posteriori, respectively, error analyses are discussed. The nonlinear case, that is \eqref{eq:NSEStrong} is considered in \cite{NSE:2020} where existence and uniqueness for small data, and in the same functional setting, is proved. In the setting of uniqueness, an a priori error analysis for a numerical scheme is also developed. This brings us to this work and its contributions. The solution to \eqref{eq:NSEStrong}, because of the singular data, is not expected to be smooth, and thus adaptive methods must be developed to efficiently approximate it. Our goal here is to develop and analyze a reliable and efficient a posteriori error estimator, and show its performance when used in a standard adaptive procedure.

To achieve these goals we organize our presentation as follows. We set notation in section \ref{sec:prelim}, where we also recall the definition of Muckenhoupt weights and introduce the weighted spaces we shall work with. In section \ref{sec:Navier_Stokes}, we introduce a suitable weak formulation
for problem (3) in weighted spaces and review existence and uniqueness results for small data. Section
\ref{sec:inf_sup_stable_a_posteriori} presents basic ingredients of finite element methods. Section \ref{sec:a_posteriori} is one of the highlights of our work. In section \ref{subsec:Ritz} we introduce a Ritz projection of the residuals and prove, in section \ref{subsec:upper_bound}, that the energy norm of the error can be bounded in terms of the the energy norm of the Ritz projection. We thus propose, in section \ref{subsec:residual--type}, an a posteriori error estimator for inf--sup stable finite element approximations of problem \eqref{eq:NSEStrong}; the devised error estimator is proven to be locally efficient and
globally reliable. We conclude, in Section \ref{sec:numerics}, with a series of numerical experiments that illustrate our theory.

\section{Notation and preliminaries}
\label{sec:prelim}

Let us set notation and describe the setting we shall operate with.  Throughout this work $\Omega\subset\mathbb{R}^2$ is an open and bounded polygonal domain with Lipschitz boundary $\partial\Omega$. Notice that we do not assume that $\Omega$ is convex. If $\mathcal{W}$ and $\mathcal{Z}$ are Banach function spaces, we write $\mathcal{W} \hookrightarrow \mathcal{Z}$ to denote that $\mathcal{W}$ is continuously embedded in $\mathcal{Z}$. We denote by $\mathcal{W}'$ and $\|\cdot\|_{\mathcal{W}}$ the dual and the norm of $\mathcal{W}$, respectively.

For $E \subset \bar\Omega$ of finite Hausdorff $i$-dimension, $i \in \{1,2\}$, we denote its measure by $|E|$. If $E$ is such a set and $f : E \to \R$ we denote its mean value by
\[
 \fint_E f  = \frac{1}{|E|}\int_{E} f .
\]

The relation $a \lesssim b$ indicates that $a \leq C b$, with a constant $C$ that depends neither on $a$, $b$ nor the discretization parameters. The value of $C$ might change at each occurrence.

\subsection{Weights} 
A notion which will be fundamental for further discussions is that of a weight. By a weight we mean a locally integrable, nonnegative function defined on $\mathbb{R}^2$. Of particular interest in our constructions will be weights that belong to the Muckenhoupt class $A_2$ \cite{MR1800316}, which consist of all weights $\omega$ such that
\begin{equation}
\label{A_pclass}
\left[ \omega \right]_{A_2} := \sup_{B} \left( \fint_{B} \omega \right) \left( \fint_{B} \omega^{-1} \right)  < \infty, 
\end{equation}
where the supremum is taken over all balls $B$ in $\R^2$. For $\omega \in A_2$ the quantity $[\omega]_{A_2}$ is the Muckenhoupt characteristic of $\omega$. We refer the reader to \cite{MR1800316,MR2491902,NOS3,MR1774162} for basic facts about the class $A_2$. An $A_2$ weight which will be essential for our subsequent developments is the following. Let $z \in \Omega$ and $\alpha \in (-2,2)$. Then
\begin{equation}
\label{distance_A2}
\dist^\alpha(x) = |x-z|^{\alpha} \in A_2.
\end{equation}

An important property of the weight $\dist^\alpha$ is that there is a neighborhood of $\partial \Omega$ where $\dist^\alpha$ is strictly positive, and continuous. This observation motivates us to define a restricted class of Muckenhoupt weights \cite[Definition 2.5]{MR1601373}.

\begin{definition}[class $A_2(\Omega)$]
\label{def:ApOmega}
Let $\Omega \subset \R^2$ be a Lipschitz domain. We say that $\omega \in A_2$ belongs to $A_2(\Omega)$ if there is an open set $\calG \subset \Omega$, and $\varepsilon, \omega_l >0$ such that:
\[
\{ x \in \Omega: \mathrm{dist}(x,\partial\Omega)< \varepsilon\} \subset \calG,
\qquad
\omega_{|\bar\calG} \in C(\bar\calG),
\qquad
\omega_l \leq \omega(x) \quad \forall x \in \bar\calG.
\]
\end{definition}

As we have mentioned in the introduction, what allows us to consider rough forcings in \eqref{eq:NSEStrong} is the use of weights and weighted spaces, as we will define below. We must note, however, that since we are not assuming our polygonal domain $\Omega$ to be convex, the same considerations given in the counterexample of \cite[page 2]{DO:17} show that we cannot work with general weights, and we cannot allow our forcings to have singularities near the boundary. This is the importance of the class $A_2(\Omega)$ of Definition~\ref{def:ApOmega}.

\subsection{Weighted spaces} 
Let $E$ be an arbitrary domain in $\R^2$ and $\omega \in A_2$. We define $L^2(\omega,E)$ as the space of Lebesgue measurable functions in $E$ such that 
\[
 \| v \|_{L^2(\omega,\Omega)} = \left(\int_{E} \omega |v|^2 \right )^{\frac{1}{2}} < \infty.
\]
We define the weighted Sobolev space $H^1(\omega,E)$ as the set of functions $v \in L^2(\omega,E)$ such that, for every multiindex $\alpha \in \mathbb{N}_0^2$ with $|\alpha| \leq 1$ we have that the distributional derivatives $D^{\alpha} v \in L^2(\omega,E)$. We endow $H^1(\omega,E)$ with the norm
\begin{equation}
\label{eq:norm}
 \| v \|_{H^1(\omega,E)}:= \left(  \| v \|_{L^2(\omega,E)}^2 +  \| \nabla v \|_{L^2(\omega,E)}^2  \right)^{\frac{1}{2}}.
\end{equation}
We define $H_0^1(\omega,E)$ as the closure of $C_0^{\infty}(E)$ in $H^1(\omega,E)$. We notice that, owing to a weighted Poincar\'e inequality \cite{MR643158}, over $H_0^1(\omega,E)$ the seminorm $\| \nabla v \|_{L^2(\omega,E)}$ is equivalent to the norm defined in \eqref{eq:norm}.

Spaces of vector valued functions will be denoted by boldface, that is 
\[
 \bH_0^1(\omega,E) = [ H_0^1(\omega,E) ]^{2}, \quad \| \GRAD \bv \|_{\bL^2(\omega,E)}:= \left( \sum_{i=1}^{2} \| \nabla v_i \|^2_{L^2(\omega,E)} \right)^{\frac{1}{2}},
\]
where $\bv = (v_1,v_2)^\intercal$.

The following product spaces with the weight $\dist^{\alpha}$ will be of particular importance. For $\alpha \in (-2,2)$, we define
\begin{equation}
\label{XandY}
\mathcal{X}(E) = \bH^{1}_0(\dist^{\alpha},E) \times L^2(\dist^{\alpha},E)/ \mathbb{R}, \quad
\mathcal{Y}(E) = \bH^{1}_0(\dist^{-\alpha},E) \times L^2(\dist^{-\alpha},E)/ \mathbb{R},
\end{equation}
which we endow with standard product space norms. 
When $E = \Omega$, and in order to simplify the presentation of the material, we write $\mathcal{X} = \mathcal{X}(\Omega)$ and $\mathcal{Y} = \mathcal{Y}(\Omega)$.

\section{The stationary Navier Stokes equations under singular forcing}
\label{sec:Navier_Stokes}

For $\alpha \in (-2,2)$, we define the bilinear forms
\begin{equation}
\label{eq:defofforma}
\begin{aligned}
a : \bH^1_0(\dist^{\alpha},\Omega) \times \bH^1_0(\dist^{-\alpha},\Omega) \to \R, 
\quad
a(\bw,\bv) := \int_\Omega \nabla \bw : \nabla \bv,
\end{aligned}
\end{equation}
and
\begin{equation}
\label{eq:defofformb}
 \begin{aligned}
b_\pm &: \bH^1_0(\dist^{\pm\alpha},\Omega)\times L^2(\dist^{\mp\alpha},\Omega) \to \R, 
\quad
b_\pm(\bv,q) :=  - \int_{\Omega} q \DIV \bv.
\end{aligned}
\end{equation}
We also define the trilinear form
\begin{equation}
\label{eq:defofformc}
 \begin{aligned}
c : [\bH^1_0(\dist^{\alpha},\Omega) ]^2 \times \bH^1_0(\dist^{-\alpha},\Omega)  \to \R, 
\quad
c(\ue, \bw; \bv) :=  - \int_{\Omega} \ue \otimes \bw : \GRAD \bv.
\end{aligned}
\end{equation}

The results of \cite{OS:17infsup} yield an inf--sup condition for the bilinear form $a$ on weighted spaces, i.e., we have
\begin{multline}
\label{eq:infsup}
  \inf_{0 \neq \bv \in \bH^1_0(\dist^{\alpha},\Omega)} \sup_{0 \neq \bw \in \bH^1_0(\dist^{-\alpha},\Omega)} \frac{ a(  \bv,  \bw ) }{ \| \nabla \bv\|_{\bL^2(\dist^{\alpha},\Omega)} \| \nabla \bw \|_{\bL^2(\dist^{-\alpha},\Omega)} } 
  = 
  \\
    \inf_{0 \neq \bw \in \bH^1_0(\dist^{-\alpha},\Omega)} \sup_{0 \neq \bv \in \bH^1_0(\dist^{\alpha},\Omega)} \frac{ a(  \bv,  \bw ) }{ \| \nabla \bv\|_{\bL^2(\dist^{\alpha},\Omega)} \| \nabla \bw \|_{\bL^2(\dist^{-\alpha},\Omega)} } > 0.
\end{multline}

On the other hand, since we are in two dimensions and $\dist^{\alpha} \in A_2$, \cite[Theorem 1.3]{MR643158} shows that $\bH_0^1(\dist^{\alpha},\Omega) \hookrightarrow \bL^4(\dist^{\alpha},\Omega)$. Thus, if we denote by $C_{4 \to 2}$ the best embedding constant, we have that the convective term can be bound as follows:
\begin{equation}
  \begin{aligned}
  |c(\ue, \bw; \bv)| &= \left | \int_{\Omega} \ue \otimes \bw : \GRAD \bv \right | \leq \| \ue \|_{\bL^4(\dist^{\alpha},\Omega)}\| \bw \|_{\bL^4(\dist^{\alpha},\Omega)}  \| \nabla \bv \|_{\bL^2(\dist^{-\alpha},\Omega)} \\
  &\leq C_{4 \to 2}^2 \|\nabla \ue \|_{\bL^2(\dist^{\alpha},\Omega)}
  \| \nabla \bw \|_{\bL^2(\dist^{\alpha},\Omega)}  \| \nabla \bv \|_{\bL^2(\dist^{-\alpha},\Omega)}.
  \end{aligned}
\label{eq:bound_convective_term}
\end{equation}

\subsection{Weak formulation}
With definitions \eqref{eq:defofforma}--\eqref{eq:defofformc} at hand, we consider the following weak formulation for problem \eqref{eq:NSEStrong}: Find $(\ue,\pe) \in \calX$ such that
\begin{equation}
\label{eq:NSEVar}
a(\ue,\bv) + b_-(\bv,\pe) + c(\ue,\ue;\bv)
% \int_\Omega \left(  \GRAD  \ue : \GRAD \bv - \ue \otimes \ue : \GRAD \bv - \pe \DIV \bv\right)  
= \langle \bF \delta_z, \bv \rangle,
\qquad
b_+(\ue,q) = 0,
\end{equation}
for all $(\bv,q) \in \mathcal{Y}$. Here and in what follows, $\langle \cdot, \cdot \rangle$ denotes a duality pairing. The spaces used for such pairing shall be evident from the context.
We must immediately comment that, in order to guarantee that $\delta_{z} \in H_0^1(\dist^{-\alpha},\Omega)'$, and thus that $\langle \bF \delta_z, \bv \rangle$ is well--defined for $\bv \in \bH^1_0(\dist^{-\alpha},\Omega)$, the parameter $\alpha$ should be restricted to belong to the interval $(0,2)$; see \cite[Lemma 7.1.3]{KMR} and \cite[Remark 21.18]{MR2305115}.

\subsection{Existence and uniqueness for small data}
Let us define the mappings $\mathcal{S}: \calX \to \calY'$, $\mathcal{NL}: \calX \to \calY'$, and $\mathcal{F} \in \calY'$ by
\begin{align*}
  \langle \mathcal{S}(\bu,p), (\bv,q) \rangle & = a(\ue,\bv) + b_-(\bv,p) + b_+(\ue,q), 
\\
  \langle \mathcal{NL}(\bu,p), (\bv,q) \rangle & = 
  c(\ue,\ue;\bv),
\end{align*}
and $\langle \mathcal{F}, (\bv,q) \rangle = \langle \bF \delta_z, \bv \rangle$, respectively. With this notation \eqref{eq:NSEVar} can be equivalently written as the following operator equation in $\mathcal{Y}'$:
\[
  \mathcal{S}(\ue,p) + \mathcal{NL}(\ue,p) = \mathcal{F}.
\]

In what follows, by $\|\mathcal{S}^{-1} \|$ we shall denote the $\calY' \to \calX$ norm of $\mathcal{S}^{-1}$. We recall that $C_{4 \to 2}$ denotes the best constant in the embedding $\bH_0^1(\dist^{\alpha},\Omega) \hookrightarrow \bL^4(\dist^{\alpha},\Omega)$. Assume that the forcing term $ \bF \delta_z$ is sufficiently small so that
\begin{equation}
\label{eq:smallness}
C_{4\to2}^2 \| \mathcal{S}^{-1} \|^2  \| \bF \delta_z \|_{\bH^{1}_0(\dist^{-\alpha},\Omega)'} < \frac16.
\end{equation}
With this assumption at hand, we have existence and uniqueness for small data. 

\begin{proposition}[existence and uniqueness]
\label{cor:corContractionSmall}
Let $\Omega$ be Lipschitz. Assume that the forcing term $ \bF \delta_z$ is sufficiently small so that \eqref{eq:smallness} holds. If $\alpha \in (0,2)$, then there is a unique solution of \eqref{eq:NSEVar}. Moreover, this solution satisfies the estimates
\begin{equation}
  \| \GRAD \ue \|_{\bL^2(\dist^{\alpha},\Omega)} \leq \frac32 \| \mathcal{S}^{-1}  \| \ \|  \bF \delta_z \|_{\bH^{1}_0(\dist^{-\alpha},\Omega)'}
\label{eq:estimate_NSE}
\end{equation}
and
\[
  \| \pe \|_{L^2(\dist^\alpha,\Omega)} \lesssim \| \GRAD \ue \|_{\bL^2(\dist^{\alpha},\Omega)} + \| \GRAD \ue \|_{\bL^2(\dist^{\alpha},\Omega)}^2 + \|  \bF \delta_z \|_{\bH^{1}_0(\dist^{-\alpha},\Omega)'}.
\]
where the hidden constant is independent of $\ue$, $\pe$, and $\bF \delta_z$.
\end{proposition}
\begin{proof}
Existence, uniqueness, and the velocity estimate are the content of \cite[Corollary 1]{NSE:2020}.

To show the estimate on the pressure, we invoke the weighted inf--sup condition \cite[Theorem 3.1]{MR2731700}, \cite[Theorem 1]{MR2548872}, \cite[Lemma 15]{DuranOtarolaAJS}
\[
  \| p \|_{L^2(\dist^\alpha,\Omega)} \lesssim \sup_{\boldsymbol0 \neq \bv \in \bH^1_0(\dist^{-\alpha},\Omega)} \frac{ b_-(\bv,p) }{\| \GRAD \bv \|_{\bL^2(\dist^{-\alpha},\Omega)}} \quad \forall p \in L^2(\dist^\alpha,\Omega)/\R.
\]
The hidden constant depends only on $\Omega$ and $[\dist^\alpha]_{A_2}$. Using this estimate for $\pe$, the first equation in \eqref{eq:NSEVar}, and the estimate on the convective term of \eqref{eq:bound_convective_term}
yields the desired pressure estimate.
\end{proof}

\section{Discretization}
\label{sec:inf_sup_stable_a_posteriori}

We now propose a finite element scheme to approximate the solution to \eqref{eq:NSEVar}. To accomplish this task, we first introduce some terminology and a few basic ingredients.

\subsection{Triangulation}
\label{sec:fem}
We denote by $\T = \{T\}$ a conforming partition of $\bar\Omega$ into closed simplices $T$ with size $h_T = \diam(T)$ and define $h_{\T} = \max_{T \in \T} h_T$.  We denote by $\Tr$ the collection of conforming and shape regular meshes that are refinements of an initial mesh $\T_0$ \cite{CiarletBook,Guermond-Ern}.

We denote by $\Sides$ the set of internal one dimensional interelement boundaries $S$ of $\T$. For $S \in \Sides$, we indicate by $h_S$ the diameter of $S$. If $T\in\T$, we define $\Sides_{T}$ as the subset of $\Sides$ that contains the sides of $T$. For $S \in \Sides$, we set $\Ne_S = \{ T^+, T^-\}$, where $T^+, T^- \in \T$ are such that $S=T^+ \cap T^-$.
% , in other words, $\Ne_S$ denotes the subset of $\T$ that contains the two elements of $\T$ that have $S$ as a side. 
For $T \in \T$, we define the following \emph{stars} or \emph{patches} associated with the element $T$
\begin{equation}
\Ne_T := \left\{ T' \in \T : \Sides_T \cap \Sides_{T'} \neq \emptyset \right\},
\qquad
\mathcal{S} _T := \{ T' \in \T : T \cap T' \neq \emptyset \}.
\label{eq:NeTstar}
\end{equation}
  
\subsection{Finite element spaces}
\label{sec:fem_spaces}

Given a mesh $\T \in \Tr$, we denote by $\mathbf{V}(\T)$ and $\mathcal{P}(\T)$ the finite element spaces that approximate the velocity field and the pressure, respectively, constructed over $\T$. The following elections are popular.
\begin{enumerate}[(a)]
\item The \emph{mini} element. This pair is studied, for instance, in \cite{MR799997}, \cite[Section 4.2.4]{Guermond-Ern}, and it is defined by
\begin{align}
\label{eq:mini_V}
\mathbf{V}(\T) & = \left\{  \bv_{\T} \in \mathbf{C}(\bar \Omega):\ \forall T \in \T, \bv_{\T|T} \in [\mathbb{P}_1(T) \oplus \mathbb{B}(T)]^2 \right\} \cap \bH_0^1(\Omega),
\\
\label{eq:mini_P}
\mathcal{P}(\T) & = \left\{  q_{\T} \in  L^2(\Omega)/\R \cap C(\bar \Omega):\ \forall T \in \T, q_{\T|T} \in \mathbb{P}_1(T) \right\},
\end{align}
where $\mathbb{B}(T)$ denotes the space spanned by a local bubble function.

\item The Taylor--Hood pair. The lowest order Taylor--Hood element \cite{hood1974navier}, \cite{MR993474}, \cite[Section 4.2.5]{Guermond-Ern} is defined by
 \begin{align}
\label{eq:th_V}
\mathbf{V}(\T) & = \left\{  \bv_{\T} \in \mathbf{C}(\bar \Omega): \ \forall T \in \T, \bv_{\T|T} \in \mathbb{P}_2(T)^2 \right\} \cap \bH_0^1(\Omega),
\\
\label{eq:th_P}
\mathcal{P}(\T) & = \left\{ q_{\T} \in L^2(\Omega)/\R \cap C(\bar \Omega):\  \forall T \in \T, q_{\T|T} \in \mathbb{P}_1(T) \right\}.
\end{align}
\end{enumerate}

It is important to observe that, if $\omega \in A_2$, we have, for the elections given by \eqref{eq:mini_V}--\eqref{eq:th_P},
\[
\mathbf{V}(\T) \subset \mathbf{W}_0^{1,\infty}(\Omega) \subset \bH^1_0(\omega,\Omega), \qquad \mathcal{P}(\T) \subset L^{\infty}(\Omega)/\R \subset L^2(\omega,\Omega)/\R.
\] 
In addition, these spaces are compatible, in the sense that they satisfy weighted versions of the classical LBB condition \cite{Guermond-Ern,MR851383}. Namely, there exists a positive constant $\beta>0$, which is independent of $\T$ and for which we have, \cite[Theorems 6.2 and 6.4]{DuranOtarolaAJS}
\begin{equation}
\label{eq:infsup_discrete}
\beta \|q_{\T} \|_{L^2(\dist^{\pm\alpha},\Omega)} \leq\sup_{ \boldsymbol0 \neq \bv_{\T} \in \mathbf{V}(\T) } \frac{b_\mp(\bv_{\T},q_{\T}) }{ \| \nabla \bv_{\T} \|_{\bL^2(\dist^{\mp\alpha},\Omega)}} \quad \forall q_{\T} \in \mathcal{P}(\T).
\end{equation}

\subsection{Finite element approximation}
\label{sec:fem_approximation}
We now define a finite element approximation of problem \eqref{eq:NSEVar} as follows: Find $(\bu_{\T},\pe_{\T}) \in \mathbf{V}(\T) \times \mathcal{P}(\T)$ such that
\begin{equation}
\begin{aligned}
\label{eq:NSEh}
a(\bu_{\T},\bv_{\T}) + b_-(\bv_{\T}, \pe_{\T}) + c( \bu_{\T}, \bu_{\T};  \bv_{\T}) = \bF \cdot \bv_{\T}(z),
\quad
b_+(\bu_{\T},q_{\T}) = 0,
\end{aligned}
\end{equation}
for all $\bv_{\T} \in \mathbf{V}(\T)$ and $q_{\T} \in \mathcal{P}(\T)$.

Denote by $\mathcal{S}_{\T}$ the discrete version of $\mathcal{S}$. Since the pairs $(\mathbf{V}(\T),\mathcal{P}(\T))$ satisfy all the assumptions of \cite{DuranOtarolaAJS}, we have that the Stokes projection onto $(\mathbf{V}(\T),\mathcal{P}(\T))$ is stable in $\bH^1_0(\dist^{\alpha},\Omega) \times L^2(\dist^{\alpha},\Omega)/\R$. As a consequence, $\mathcal{S}_{\T}$ is a bounded linear operator whose inverse $\mathcal{S}_{\T}^{-1}$ is bounded uniformly over all $\T \in \Tr$; see \cite[Theorem 4.1]{DuranOtarolaAJS}. Assume that the forcing term $\bF \delta_z$ is sufficiently small so that \eqref{eq:smallness} holds with $\mathcal{S}$ replaced with $\mathcal{S}_{\T}$. Then there is a unique $(\bu_{\T},\pe_{\T}) \in \mathbf{V}(\T) \times \mathcal{P}(\T)$ that solves \eqref{eq:NSEh}. Moreover, we have 
%an estimate similar to \eqref{eq:estimate_NSE}.
\begin{equation}
  \| \GRAD \ue_{\T} \|_{\bL^2(\dist^{\alpha},\Omega)} \leq \frac32 \| \mathcal{S}_{\T}^{-1}  \| \ \|  \bF \delta_z \|_{\bH^{1}_0(\dist^{-\alpha},\Omega)'},
\label{eq:estimate_NSEh}
\end{equation}
with a pressure estimate similar to that of Proposition~\ref{cor:corContractionSmall}.

\subsection{A quasi--interpolation operator}
\label{subsec:quasi}
In order to derive reliability properties for the proposed a posteriori error estimator, it is useful to have at hand a suitable quasi--interpolation operator with optimal approximation properties \cite{Verfurth}. We consider the operator $\Pi_{\T}:\bL^1(\Omega) \rightarrow \mathbf{V}(\T)$ analyzed in \cite{NOS3}. The construction of $\Pi_{\T}$ is inspired in the ideas developed by Cl\'ement \cite{MR0400739}, Scott and Zhang \cite{MR1011446}, and Dur\'an and Lombardi \cite{MR2164092}: it is built on local averages over stars and thus well--defined for functions in $\bL^1(\Omega)$; it also exhibits optimal approximation properties. In what follows, we shall make use of the following estimates of the local interpolation error \cite{MR3892359,NOS3}. To present them, we first define, for $T \in \T$, 
\begin{equation}
\label{eq:DT}
D_T: = \max_{x \in T} |x-z|.
\end{equation}

\begin{proposition}[stability and interpolation estimates]
Let $\alpha \in (-2,2)$, and $T \in \T$. Then, for every $\bv \in \bH^1(\dist^{\pm\alpha},\mathcal{S}_T)$, we have the local stability bound
\begin{equation}
\label{eq:local_stability}
 \| \GRAD \Pi_{\T} \bv\|_{\bL^2(\dist^{\pm\alpha},T)} \lesssim \| \GRAD \bv \|_{\bL^2(\dist^{\pm\alpha},\mathcal{S}_T)}
\end{equation}
and the interpolation error estimate
\begin{equation}
\label{eq:interpolation_estimate_1}
\|  \bv - \Pi_{\T} \bv \|_{\bL^2(\dist^{\pm\alpha},T)} \lesssim h_T \| \GRAD \bv \|_{\bL^2(\dist^{\pm\alpha},\mathcal{S}_T)}.
\end{equation}
In addition, if $\alpha \in (0,2)$, then we have that
\begin{equation}
\label{eq:interpolation_estimate_2}
 \|  \bv - \Pi_{\T} \bv \|_{\bL^2(T)} \lesssim h_T D_T^{\frac{\alpha}{2}} \| \GRAD \bv \|_{\bL^2(\dist^{-\alpha},\mathcal{S}_T)},
\end{equation}
The hidden constants, in the previous inequalities, are independent of $\bv$, the cell $T$, and the mesh $\T$.
\end{proposition}

\begin{proposition}[trace interpolation estimate]
Let $\alpha \in (0,2)$, $T \in \T$, $S \subset \Sides_T$, and $\bv \in \bH^1(\dist^{-\alpha},\mathcal{S}_T)$. Then we have the following interpolation error estimate for the trace
\begin{equation}
\label{eq:interpolation_estimate_trace}
 \|  \bv - \Pi_{\T} \bv \|_{\bL^2(S)} \lesssim h_T^{\frac{1}{2}} D_T^{\frac{\alpha}{2}} \| \GRAD \bv \|_{\bL^2(\dist^{-\alpha},\mathcal{S}_T)},
\end{equation}
where the hidden constant is independent of $\bv$, $T$, and the mesh $\T$.
\end{proposition}

\section{A posteriori error estimates}
\label{sec:a_posteriori}

In this section, we analyze a posterior error estimates for the finite element approximation \eqref{eq:NSEh} of problem \eqref{eq:NSEVar}. To begin with such an analysis, we define the velocity and pressure errors $(\be_{\ue},e_\pe)$ by
\begin{equation}
\label{eq:errors}
\boldsymbol{e}_{\ue}:=\ue-\ue_{\T} \in \bH^1_0(\dist^\alpha,\Omega),
\quad
e_{\pe}:= \pe- \pe_{\T} \in L^2(\dist^\alpha,\Omega)/\R,
\end{equation} 
respectively.

\subsection{Ritz projection}
\label{subsec:Ritz}
In order to perform a reliability analysis for the proposed a posteriori error estimator, inspired by the developments of \cite{MR1445736}, we introduce a Ritz projection $(\boldsymbol{\Phi},\psi)$ of the residuals. The pair $(\boldsymbol{\Phi},\psi)$ is defined as the solution to the following problem: Find $(\boldsymbol{\Phi},\psi) \in \mathcal{X}$ such that
\begin{equation}
  \begin{aligned}
 a(\boldsymbol{\Phi}, \bv) &=  a(\boldsymbol{e}_{\ue}, \bv) + b_-(\bv,e_{\pe}) + c(\ue, \boldsymbol{e}_{\ue}; \bv) + c(\boldsymbol{e}_{\ue},\ue_{\T}; \bv),
 \\
 (\psi,q)_{L^2(\Omega)} &= b_+(\boldsymbol{e}_{\ue},q),
 \end{aligned}
 \label{eq:Ritz}
\end{equation}
for all $(\bv,q) \in \mathcal{Y}$.

The following results yields the well--posedness of problem \eqref{eq:Ritz}.

\begin{theorem}[Ritz projection]
Problem \eqref{eq:Ritz} has a unique solution in $\mathcal{X}$. In addition, this solution satisfies the estimate
\begin{multline}
\| \nabla \boldsymbol{\Phi} \|_{\bL^2(\dist^{\alpha},\Omega)} + \| \psi \|_{L^2(\dist^{\alpha},\Omega)} \lesssim 
\| \nabla \boldsymbol{e}_{\ue} \|_{\bL^2(\dist^{\alpha},\Omega)} + \| e_{\pe} \|_{L^2(\dist^{\alpha},\Omega)} 
\\
+\| \nabla \boldsymbol{e}_{\ue} \|_{\bL^2(\dist^{\alpha},\Omega)}\left(  \| \nabla \ue \|_{\bL^2(\dist^{\alpha},\Omega)}+ \| \nabla \ue_{\T} \|_{\bL^2(\dist^{\alpha},\Omega)} \right),
 \label{eq:Ritz_Stability}
\end{multline}
where the hidden constant is independent of $(\boldsymbol{\Phi},\psi)$, $(\ue,\pe)$, and $(\ue_{\T},\pe_{\T})$.
\label{thm:Ritz_Projection}
\end{theorem}
\begin{proof}
Define
\[
 \mathfrak{G}: \bH_0^1(\dist^{-\alpha},\Omega) \rightarrow \mathbb{R}, 
 \quad
 \bv \mapsto a(\boldsymbol{e}_{\ue}, \bv) + b_-(\bv,e_{\pe}) + c(\ue, \boldsymbol{e}_{\ue}; \bv) + c(\boldsymbol{e}_{\ue},\ue_{\T}; \bv).
\]
Notice that $\mathfrak{G}$ is linear. To prove that $\mathfrak{G} \in \bH_0^1(\dist^{-\alpha},\Omega)'$, we observe that 
\begin{multline*}
| \mathfrak{G}(\bv) | \leq ( \| \nabla \boldsymbol{e}_{\ue} \|_{\bL^2(\dist^{\alpha},\Omega)} + \| e_{\pe} \|_{L^2(\dist^{\alpha},\Omega)} + \| \ue \|_{\bL^4(\dist^{\alpha},\Omega)} \| \boldsymbol{e}_{\ue} \|_{\bL^4(\dist^{\alpha},\Omega)}
\\
+
\| \boldsymbol{e}_{\ue} \|_{\bL^4(\dist^{\alpha},\Omega)} \| \ue_{\T} \|_{\bL^4(\dist^{\alpha},\Omega)})\| \nabla \bv \|_{\bL^2(\dist^{-\alpha},\Omega)}.
\end{multline*}
This, combined with the Sobolev embedding $\bH_0^1(\dist^{\alpha},\Omega)\hookrightarrow \bL^4(\dist^{\alpha},\Omega) $ 
allows us to conclude.

Since $\dist^{\alpha} \in A_2(\Omega)$ and $\mathfrak{G} \in \bH_0^1(\dist^{-\alpha},\Omega)'$, we can thus invoke the results of \cite{OS:17infsup} to conclude the existence and uniqueness of $\boldsymbol{\Phi} \in \bH_0^1(\dist^{\alpha},\Omega)$ together with the bound
\begin{multline}
\| \nabla \boldsymbol{\Phi} \|_{\bL^2(\dist^{\alpha},\Omega)} \lesssim   \| \nabla \boldsymbol{e}_{\ue} \|_{\bL^2(\dist^{\alpha},\Omega)} + \| e_{\pe} \|_{L^2(\dist^{\alpha},\Omega)} 
\\
+\| \nabla \boldsymbol{e}_{\ue} \|_{\bL^2(\dist^{\alpha},\Omega)}(  \| \nabla \ue \|_{\bL^2(\dist^{\alpha},\Omega)}+ \| \nabla \ue_{\T} \|_{\bL^2(\dist^{\alpha},\Omega)}).
\end{multline}

On the other hand, since $\boldsymbol{e}_{\ue} \in \bH_0^1(\dist^{\alpha},\Omega)$, $b_+(\boldsymbol{e}_{\ue}, \cdot )$ defines a linear and bounded functional in $L^2(\dist^{-\alpha},\Omega)/\R$. This immediately yields the existence and uniqueness of $\psi \in L^2(\dist^{\alpha},\Omega)/\R$ together with the estimate
\[
 \| \psi \|_{L^2(\dist^{\alpha},\Omega)} \leq \| \DIV \boldsymbol{e}_{\ue} \|_{\bL^2(\dist^{\alpha},\Omega)}.
\]

A collection of the derived estimates yields \eqref{eq:Ritz_Stability}. This concludes the proof.
\end{proof}

\subsection{An upper bound for the error}
\label{subsec:upper_bound}
With the results of Theorem \ref{thm:Ritz_Projection} at hand, we observe that the pair $(\boldsymbol{e}_{\ue}, e_{\pe})$ can be seen as the solution to the following Stokes problem: Find $(\boldsymbol{e}_{\ue},e_{\pe})\in \mathcal{X}$ such that
\begin{equation}
  a(\boldsymbol{e}_{\ue}, \bv) + b_-(\bv,e_{\pe}) = \mathfrak{F}(\bv),
\quad
b_+(\boldsymbol{e}_{\ue},q) = (\psi,q)_{L^2(\Omega)}
 \label{eq:eu_ep_solves}
\end{equation}
for all $(\bv,q) \in \mathcal{Y}$, where
\[
 \mathfrak{F}: \bH_0^1(\dist^{-\alpha},\Omega) \rightarrow \mathbb{R}, 
 \quad
 \bv \mapsto a(\boldsymbol{\Phi}, \bv) - c(\ue, \boldsymbol{e}_{\ue}; \bv) - c(\boldsymbol{e}_{\ue},\ue_{\T}; \bv).
\]
It is clear that $\mathfrak{F}$ is linear in $\bH_0^1(\dist^{-\alpha},\Omega)$. In fact, $\mathfrak{F} \in \bH_0^1(\dist^{-\alpha},\Omega)'$, since
\begin{equation}
  \begin{aligned}
 \|\mathfrak{F}\|_{\bH_0^1(\dist^{-\alpha},\Omega)'} &\leq \| \nabla \boldsymbol{\Phi} \|_{\bL^2(\dist^{\alpha},\Omega)}
 \\
 &+ C_{4\to2}^2 \| \nabla \boldsymbol{e}_{\ue} \|_{\bL^2(\dist^{\alpha},\Omega)} \left( \| \nabla \ue \|_{\bL^2(\dist^{\alpha},\Omega)} + \| \nabla \ue_{\T} \|_{\bL^2(\dist^{\alpha},\Omega)} \right).
  \end{aligned}
\label{eq:estimate_for_F}
\end{equation}

With the aid of this identification, we now prove that the energy norm of the error can be bounded in terms of the the energy norm of the Ritz projection, which in turn will allow us to provide computable upper bounds for the error. To do so, we must assume that the forcing term $\bF \delta_z$ is sufficiently small so that
\begin{equation}
1 - \| \mathcal{S}^{-1} \| C_{4\to2}^2 \left[  \| \nabla \ue \|_{\bL^2(\dist^{\alpha},\Omega)} + \| \nabla \ue_{\T} \|_{\bL^2(\dist^{\alpha},\Omega)} \right]  \geq \lambda > 0.
\label{eq:smallness_assumption}
\end{equation}

\begin{corollary}[upper bound for the error]
Assume that the forcing term $ \bF \delta_z$ is sufficiently small so that \eqref{eq:smallness_assumption} holds. We then have that
\begin{equation}
\| \nabla \boldsymbol{e}_{\ue} \|_{\bL^2(\dist^{\alpha},\Omega)} + \| e_{\pe} \|_{L^2(\dist^{\alpha},\Omega)} 
\lesssim 
\| \nabla \boldsymbol{\Phi} \|_{\bL^2(\dist^{\alpha},\Omega)} +  \| \psi \|_{L^2(\dist^{\alpha},\Omega)} ,
\label{eq:upper_bound_for_error}
\end{equation}
where the hidden constant is independent of $(\ue,\pe)$, $(\ue_{\T},\pe_{\T})$, and $(\boldsymbol{\Phi},\psi)$.
\label{thm:upper_bound_for_error}
\end{corollary}
\begin{proof}
Since $\dist^{\alpha} \in A_2(\Omega)$ and $\mathfrak{F} \in \bH_0^1(\dist^{-\alpha},\Omega)'$, we can apply \cite[Theorem 17]{OS:17infsup} to conclude that
\begin{multline*}
 \| \nabla \boldsymbol{e}_{\ue} \|_{\bL^2(\dist^{\alpha},\Omega)} + \| e_{\pe} \|_{L^2(\dist^{\alpha},\Omega)} \leq \| \mathcal{S}^{-1} \| \left( 
 \| \nabla \boldsymbol{\Phi} \|_{\bL^2(\dist^{\alpha},\Omega)} +  \| \psi \|_{L^2(\dist^{\alpha},\Omega)} \right.
 \\
\left. + C_{4\to2}^2 \| \nabla \boldsymbol{e}_{\ue} \|_{\bL^2(\dist^{\alpha},\Omega)} \left[ \| \nabla \ue \|_{\bL^2(\dist^{\alpha},\Omega)} + \| \nabla \ue_{\T} \|_{\bL^2(\dist^{\alpha},\Omega)} \right] \right),
\end{multline*}
where we have also used estimate \eqref{eq:estimate_for_F}. The smallness assumption \eqref{eq:smallness_assumption} allows us to absorb the last term in this estimate on the left hand side and obtain \eqref{eq:upper_bound_for_error}. This concludes the proof.
\end{proof}

\subsection{A residual--type error estimator}
\label{subsec:residual--type}
In this section,
we propose an a posteriori error estimator for the finite element approximation \eqref{eq:NSEh} of problem \eqref{eq:NSEVar}. 

Define, for $\alpha \in (0,2)$ and $T \in \T$, the \emph{element indicator}
\begin{multline}
\E_{\alpha}^2(\bu_{\T},\pe_{\T};T):= h_T^2D_T^{\alpha}   \|  \Delta \bu_{\T} -  (\ue_{\T} \cdot \nabla) \ue_{\T} - \DIV \ue_{\T} \ue_{\T} - \nabla \pe_{\T} \|_{\bL^2(T)}^2 
 \\
 +  \|  \DIV \bu_{\T} \|_{L^2(\dist^\alpha, T)}^2 + h_T D_T^{\alpha} \| \llbracket (\GRAD \bu_\T -\pe_{\T} \mathbf{I}) \cdot \boldsymbol{\nu}\rrbracket \|_{\bL^2(\partial T \setminus \partial \Omega)}^{2} + h_{T}^{\alpha} | \bF |^2 \#(\{z\} \cap T),
\label{eq:local_indicator}
\end{multline}
where $(\bu_{\T},\pe_{\T})$ denotes the solution to the discrete problem \eqref{eq:NSEh}, $\mathbf{I} \in \R^{d \times d}$ denotes the identity matrix, and, for a set $E$, by  $\#(E)$ we mean its cardinality. Thus $\#(\{z\} \cap T)$ equals one if $z \in T$ and zero otherwise. Here we must recall that we consider our elements $T$ to be closed sets. For a discrete tensor valued function $\bW_{\T}$, we denote by $\llbracket \bW_{\T} \cdot \boldsymbol{\nu}\rrbracket$ the jump or interelement residual, which is defined, on the internal side $S \in \Sides$ shared by the distinct elements $T^+$, $T^{-} \in \mathcal{N}_S$, by
\begin{equation}
\label{eq:jump}
 \llbracket \bW_{\T} \cdot \boldsymbol{\nu} \rrbracket =  \bW_{\T}|_{T^+}\cdot \boldsymbol{\nu}^+ +  \bW_{\T}|_{T^-} \cdot \boldsymbol{\nu}^-.
\end{equation}
Here $\boldsymbol{\nu}^+, \boldsymbol{\nu}^-$ are unit normals on $S$ pointing towards $T^+$, $T^{-}$, respectively. The \emph{error estimator} is thus defined as
\begin{equation}
\E_{\alpha}(\bu_{\T},\pe_{\T};\T): = \left( \sum_{T \in \T} \E^2_{\alpha}(\bu_{\T},\pe_{\T};T) \right)^{\frac{1}{2}}.
\label{eq:global_estimator}
\end{equation}

\subsection{Reliability}

We present the following global reliability result.

\begin{theorem}[global reliability]
Let $(\ue,\pe) \in \mathcal{X}$ be the solution to \eqref{eq:NSEVar} and the pair $(\ue_{\T},\pe_{\T}) \in (\mathbf{V}(\T), \mathcal{P}(\T))$ be its finite element approximation defined as the solution to \eqref{eq:NSEh}. Assume that the forcing term $\bF \delta_z$ is sufficiently small so that \eqref{eq:smallness_assumption} holds. If $\alpha \in (0,2)$, then
\begin{equation}
\| \nabla \boldsymbol{e}_{\ue} \|_{\bL^2(\dist^{\alpha},\Omega)} + \| e_{\pe} \|_{L^2(\dist^{\alpha},\Omega)} 
\lesssim 
\E_{\alpha}(\bu_{\T},\pe_{\T};\T),
\label{eq:reliability}
\end{equation}
where the hidden constant is independent of the continuous and discrete solutions, the size of the elements in the mesh $\T$, and $\#\T$.

\end{theorem}

\begin{proof}
We proceed in three steps.

\noindent \emph{Step 1.} Using the first equation of \eqref{eq:Ritz} and \eqref{eq:NSEVar} we obtain that
\begin{equation}
  a(\boldsymbol{\Phi}, \bv) = \langle \bF \delta_z, \bv \rangle - \sum_{T \in \T} \int_{T} \left( \nabla \ue_{\T}:\nabla \bv -  \ue_{\T} \otimes  \ue_{\T}: \nabla \bv - \pe_{\T} \DIV \bv \right) ,
  \label{eq:first_step}
\end{equation}
for every $\bv \in \bH_0^1(\dist^{-\alpha},\Omega)$. Integrating by parts we arrive at the identity
\begin{multline}
a(\boldsymbol{\Phi}, \bv) = \langle \bF \delta_z, \bv \rangle + \sum_{S \in \Sides} \int_S \llbracket (\nabla \bu_{\T}  - p_{\T} \mathbf{I}) \cdot \boldsymbol{\nu}\rrbracket \cdot \bv
\\
+ \sum_{T \in \T} \int_{T} \left(\Delta \ue_{\T} -  (\ue_{\T} \cdot \nabla) \ue_{\T} - \DIV \ue_{\T} \ue_{\T} - \nabla \pe_{\T} \right)
\cdot \bv.
\label{eq:indentity_Phi_residual_new}
\end{multline}

On the other hand, the first equation of problem \eqref{eq:NSEh} can be rewritten as
\[
\langle \bF \delta_z, \bv_{\T} \rangle - a(\ue_{\T},\bv_{\T}) - b_-(\bv_{\T}, \pe_{\T}) - c(\ue_{\T},\ue_{\T}; \bv_{\T}) = 0 \quad
\forall \bv_{\T} \in \mathbf{V}(\T).
\]
Set $\bv_{\T} = \Pi_{\T} \bv$ in the previous expression, apply, again, an integration by parts formula and invoke \eqref{eq:first_step} to arrive at the identity
\begin{multline}
a(\boldsymbol{\Phi}, \bv) = \langle \bF \delta_z, \bv-\Pi_{\T}\bv  \rangle + \sum_{S \in \Sides} \int_S \llbracket (\nabla \bu_{\T}  - \pe_{\T} \mathbf{I}) \cdot \boldsymbol{\nu}\rrbracket \cdot (\bv-\Pi_{\T}\bv)
\\
+ \sum_{T \in \T} \int_{T} \left(\Delta \ue_{\T} -  (\ue_{\T} \cdot \nabla) \ue_{\T} - \DIV \ue_{\T} \ue_{\T} - \nabla \pe_{\T} \right) \cdot
(\bv-\Pi_{\T}\bv) = \mathbf{I} + \mathbf{II} + \mathbf{III}.
\label{eq:indentity_Phi_residual}
\end{multline}
Notice that, to derive the previous expression, we have used that $\int_{S} \llbracket \ue_{\T} \otimes \ue_{\T} \cdot \boldsymbol{\nu} \rrbracket \cdot (\bv-\Pi_{\T}\bv) = 0$, which follows from the fact our finite element velocity space consists of continuous functions.

We now control each term separately. To control the term $\mathbf{I}$, we first invoke the local bound of \cite[Theorem 4.7]{AGM} for $\delta_z$ and then the interpolation error estimate \eqref{eq:interpolation_estimate_1} and the stability bound \eqref{eq:local_stability} to arrive at
\begin{equation}
\begin{aligned}
 \mathbf{I} & \lesssim |\bF| \left( h_T^{\frac{\alpha}{2} - 1} \| \bv-\Pi_{\T}\bv\|_{\bL^2(\dist^{-\alpha},T)} + h_T^{\frac{\alpha}{2}} \| \nabla(\bv-\Pi_{\T}\bv)\|_{\bL^2(\dist^{-\alpha},T)}\right)
 \\
 & \lesssim |\bF|  h_T^{\frac{\alpha}{2} } \| \nabla \bv\|_{\bL^2(\dist^{-\alpha},\mathcal{S}_T)}.
\end{aligned}
\end{equation}
The control of $\mathbf{II}$ follows from the trace interpolation error estimate \eqref{eq:interpolation_estimate_trace},
\begin{equation}
\begin{aligned}
\mathbf{II} & \lesssim \sum_{S \in \Sides}  \| \llbracket (\nabla \bu_{\T}  - \pe_{\T} \mathbf{I}) \cdot \boldsymbol{\nu}\rrbracket \|_{\bL^2(S)} \|\bv-\Pi_{\T}\bv\|_{\bL^2(S)} 
\\
& \lesssim  \sum_{S \in \Sides} h_T^{\frac{1}{2}} D_T^{\frac{\alpha}{2}}\| \llbracket (\nabla \bu_{\T} -\pe_{\T} \mathbf{I}) \cdot \boldsymbol{\nu} \rrbracket \|_{\bL^2(S)} \| \nabla \bv \|_{\bL^2(\dist^{-\alpha},\mathcal{S}_T)}.
\end{aligned}
\end{equation}
We finally bound $\mathbf{III}$ using the error estimate \eqref{eq:interpolation_estimate_2}:
\begin{equation}
 \label{eq:II}
 \mathbf{III} \lesssim \sum_{T \in \T} h_T D_T^{\frac{\alpha}{2}} \| \Delta \ue_{\T} -  (\ue_{\T} \cdot \nabla) \ue_{\T} - \DIV \ue_{\T} \ue_{\T} - \nabla \pe_{\T}  \|_{\bL^2(T)} \| \nabla \bv \|_{\bL^2(\dist^{-\alpha},\mathcal{S}_T)}.
\end{equation}

We now apply the inf--sup condition \eqref{eq:infsup}, the identity \eqref{eq:indentity_Phi_residual} and the estimates obtained for the terms $\mathbf{I}$, $ \mathbf{II}$, and $\mathbf{III}$ to arrive at
\begin{multline*}
 \| \nabla \boldsymbol{\Phi} \|_{\bL^2(\dist^{\alpha},\Omega)}^2 \lesssim  \left(\sup_{ \boldsymbol0 \neq \bv \in \bH^1_0(\dist^{-\alpha},\Omega)} \frac{ a ( \boldsymbol{\Phi},\bv ) }{ \| \bv \|_{\bH_0^1(\dist^{-\alpha},\Omega)} } \right)^2
 \\
 \lesssim 
\sum_{T \in \T} \left( h_T D_T^{\alpha} \| \llbracket (\nabla \bu_{\T}  - \pe_{\T} \mathbf{I}) \cdot \boldsymbol{\nu}\rrbracket \|_{\bL^2(\partial T \setminus \partial \Omega)}^{2}  \right.
 \\
 \left. + h_T^2D_T^{\alpha}   \|  \Delta \bu_{\T} -  (\ue_{\T} \cdot \nabla) \ue_{\T} - \DIV \ue_{\T} \ue_{\T} - \nabla \pe_{\T} \|_{\bL^2(T)}^2 
 + h_{T}^{\alpha} | \bF |^2 \#(\{z \}\cap T)  \right).
\end{multline*}
Notice that, to obtain the last estimate, we have also used the finite overlapping property of stars.

\noindent \emph{Step 2.} Notice that since $\psi \in L^2(\dist^{\alpha},\Omega)$, then $\tilde q = \dist^{\alpha} \psi \in  L^2(\dist^{-\alpha},\Omega)$. Define now $q = \tilde q + c$, where $c \in \R$ is chosen so that $q \in L^2(\dist^{-\alpha},\Omega)/\R$. This particular choice of test function for the second equation of \eqref{eq:Ritz} yields
\begin{equation}
\| \psi \|^2_{L^2(\dist^{\alpha},\Omega)} =  b_+(\boldsymbol{e}_{\ue}, \dist^{\alpha} \psi) = -b_+(\ue_{\T}, \dist^{\alpha} \psi) \leq \| \DIV \ue_{\T} \|_{L^2(\dist^{\alpha},\Omega)} \| \psi \|_{L^2(\dist^{\alpha},\Omega)}.
\end{equation}
Consequently, $\| \psi \|_{L^2(\dist^{\alpha},\Omega)} \leq \| \DIV \ue_{\T} \|_{L^2(\dist^{\alpha},\Omega)}$.

\noindent \emph{Step 3.} In light of the smallness assumption \eqref{eq:smallness_assumption}, we may apply the estimate \eqref{eq:upper_bound_for_error} of Corollary \ref{thm:upper_bound_for_error} to write
\[
\| \nabla \boldsymbol{e}_{\ue} \|_{\bL^2(\dist^{\alpha},\Omega)} + \| e_{\pe} \|_{L^2(\dist^{\alpha},\Omega)} 
\lesssim 
\left(\| \nabla \boldsymbol{\Phi} \|_{\bL^2(\dist^{\alpha},\Omega)} +  \| \psi \|_{L^2(\dist^{\alpha},\Omega)} \right).
\]
The desired estimate \eqref{eq:reliability} thus follows from the estimates derived in steps 1 and 2. This concludes the proof.
\end{proof}

\subsection{Local efficiency analysis}

To derive efficiency properties for the local indicator $\E_{\alpha}(\bu_{\T},p_{\T};T)$ we utilize standard residual estimation techniques but on the basis of suitable bubble functions, whose construction we owe to \cite[Section 5.2]{AGM}.

Given $T \in \T$, we introduce an element bubble function $\varphi_T$ that satisfies  
$0 \leq \varphi_T \leq 1$,
\begin{equation}
\label{eq:bubble_T}
\varphi_T(z) = 0, \qquad |T| \lesssim \int_T \varphi_T, \qquad \| \GRAD \varphi_T \|_{\bL^{\infty}(R_T)} \lesssim h_T^{-1},
\end{equation}
and there exists a simplex $T^{*} \subset T$ such that $R_{T}:= \supp(\varphi_T) \subset T^{*}$. Notice that, since $\varphi_T$ satisfies \eqref{eq:bubble_T}, we have that
\begin{equation}
\label{eq:aux_bubble_T}
 \| \theta \|_{L^2(R_T)} \lesssim \left\| \varphi_T^{\frac{1}{2}} \theta \right\|_{L^2(R_T)} \quad \forall \theta \in \mathbb{P}_{5}(R_T).
\end{equation}

Given $S \in \Sides$, we introduce an edge bubble function $\varphi_S$ that satisfies $0 \leq \varphi_S \leq 1$,
\begin{equation}
\label{eq:bubble_S}
\varphi_S(z) = 0, \qquad |S| \lesssim \int_S \varphi_S, \qquad \| \GRAD \varphi_S \|_{ \bL^{\infty}(R_{S} ) } \lesssim h_S^{-1},
\end{equation}
and $R_S:= \supp(\varphi_S)$ is such that, if $\mathcal{N}_{S} = \{ T, T' \}$, there are simplices $T_{*} \subset T$ and $T_{*}' \subset T'$ such that $R_S \subset T_{*} \cup T_{*}' \subset T \cup T'$.

\begin{proposition}[estimates for bubble functions]
Let $T \in \T$ and $\varphi_T$ be the bubble function that satisfies \eqref{eq:bubble_T}. If $\alpha \in (0,2)$, then
\begin{equation}
\label{eq:chi_bubble_T}
h_T \| \GRAD (\theta \varphi_T) \|_{\bL^2(\dist^{-\alpha},T)} \lesssim D_T^{-\frac{\alpha}{2}} \| \theta \|_{L^2(T)} \quad \forall \theta \in \mathbb{P}_{5}(T).
\end{equation}
Let $S \in \Sides$ and $\varphi_S$ be the bubble function that satisfies \eqref{eq:bubble_S}. If $\alpha \in (0,2)$, then
\begin{equation}
\label{eq:chi_bubble_S}
h_T^{\frac{1}{2}} \| \GRAD (\theta \varphi_S) \|_{\bL^2(\dist^{-\alpha},\mathcal{N}_S)} \lesssim D_T^{-\frac{\alpha}{2}} \| \theta \|_{L^2(S)} \quad \forall \theta \in \mathbb{P}_3(S),
\end{equation}
where $\theta$ is extended to the elements that comprise $\mathcal{N}_S$ as a constant along the direction normal to $S$.
\end{proposition}
\begin{proof}
See \cite[Lemma 5.2]{AGM}.
\end{proof}

Having constructed these local bubble functions, the local efficiency can be shown following more or less standard arguments.

\begin{theorem}[local efficiency]\label{Th:efficiency}
Let  $(\bu,\pe) \in \mathcal{X}$ be the solution to problem \eqref{eq:NSEVar} and $(\bu_{\T},\pe_{\T}) \in \mathbf{V}(\T) \times \mathcal{P}(\T)$ its finite element approximation given as the solution to \eqref{eq:NSEh}. Assume that the forcing term $ \bF \delta_z$ is sufficiently small so that \eqref{eq:smallness_assumption} holds. If $\alpha \in (0,2)$, then
\begin{equation}
\label{eq:local_lower_bound}
\E^2_{\alpha}(\bu_{\T},\pe_{\T}; T) \lesssim \| \GRAD \be_{\bu} \|^2_{\bL^2(\dist^{\alpha},\mathcal{N}_{T})} + \|  e_{\pe} \|^2_{L^2(\dist^{\alpha},\mathcal{N}_{T})},
\end{equation}
where the hidden constant is independent of the continuous and discrete solutions, the size of the elements in the mesh $\T$, and $\#\T$.
 \end{theorem}
\begin{proof}
We estimate each contribution in \eqref{eq:local_indicator} separately, so the proof has several steps.

\noindent \emph{Step 1}. For $T \in \T$ we bound the bulk term $h_T^2D_T^{\alpha}\|  \Delta \bu_{\T} -  (\ue_{\T} \cdot \nabla) \ue_{\T} - \DIV \ue_{\T} \ue_{\T} - \nabla p_{\T} \|_{\bL^2(T)}^2$. To shorten notation, we define the functions
\[
 \mathbf{X}_T: = \left( \Delta \bu_{\T} -  (\ue_{\T} \cdot \nabla) \ue_{\T} - \DIV \ue_{\T} \ue_{\T} - \nabla p_{\T} \right)_{|T},
 \quad \boldsymbol{\Upsilon}_T:= \varphi_T \mathbf{X}_T.
\]
Since $\varphi(z) = 0$, we immediately conclude that $\boldsymbol{\Upsilon}_T(z) = \varphi_T(z) \mathbf{X}_T(z) = \boldsymbol0$. We utilize the definitions of $\mathbf{X}_T$ and $\boldsymbol{\Upsilon}_T$ and invoke \eqref{eq:aux_bubble_T} to conclude that
\begin{multline}
\|  \Delta \bu_{\T} -  (\ue_{\T} \cdot \nabla) \ue_{\T} - \DIV \ue_{\T} \ue_{\T} - \nabla p_{\T} \|_{\bL^2(T)}^2 
\\
\lesssim 
\int_{R_T} |\mathbf{X}_T|^2 \varphi_T 
= 
\int_{T} \mathbf{X}_T \cdot  \boldsymbol{\Upsilon}_T .
\label{eq:integral_XTJT}
\end{multline}
Set, in identity \eqref{eq:indentity_Phi_residual_new}, the test function $\bv = \boldsymbol{\Upsilon}_T$, and use that $\boldsymbol{\Upsilon}_T(z) = 0$, and that $\boldsymbol{\Upsilon}_{T|S} = \boldsymbol0$ for $S \in \Sides_{T}$ to arrive at
\[
\int_{T} \mathbf{X}_T \cdot  \boldsymbol{\Upsilon}_T  = a(\boldsymbol{\Phi},\boldsymbol{\Upsilon}_T ).
\]
Since $\supp \boldsymbol{\Upsilon}_T \subset T$, a local version of estimate \eqref{eq:Ritz_Stability}  then implies that
\begin{multline*}
 a(\boldsymbol{\Phi},\boldsymbol{\Upsilon}_T ) \leq \left( \| \nabla \boldsymbol{e}_{\ue} \|_{\bL^2(\dist^{\alpha},T)} + \| e_{\pe} \|_{L^2(\dist^{\alpha},T)} + \| \nabla \ue \|_{\bL^2(\dist^{\alpha},T)} \| \nabla \boldsymbol{e}_{\ue} \|_{\bL^2(\dist^{\alpha},T)} \right.
\\
+
\left. \| \nabla \boldsymbol{e}_{\ue} \|_{\bL^2(\dist^{\alpha},T)} \| \nabla \ue_{\T} \|_{\bL^2(\dist^{\alpha},T)} \right)
\| \nabla \boldsymbol{\Upsilon}_T \|_{\bL^2(\dist^{-\alpha},T)}.
\end{multline*}
Next, we utilize the smallness assumption \eqref{eq:smallness_assumption}, which implies the estimate
\[
 \| \nabla \ue \|_{\bL^2(\dist^{\alpha},\Omega)} + \| \nabla \ue_{\T} \|_{\bL^2(\dist^{\alpha},\Omega)} \leq \frac{1-\lambda}{\|\mathcal{S}^{-1} \|C_{4\to2}^2},
\]
to obtain that
\begin{equation}
  a(\boldsymbol{\Phi},\boldsymbol{\Upsilon}_T ) \lesssim \left( \| \nabla \boldsymbol{e}_{\ue} \|_{\bL^2(\dist^{\alpha},T)} + \| e_{\pe} \|_{L^2(\dist^{\alpha},T)} \right) \| \nabla \boldsymbol{\Upsilon}_T \|_{\bL^2(\dist^{-\alpha},T)}.
\label{eq:bound_for_a_local}
\end{equation}
We replace this estimate into \eqref{eq:integral_XTJT} to derive
\begin{equation}
\| \mathbf{X}_T \|_{\bL^2(T)}^2 
\lesssim \left( \| \nabla \boldsymbol{e}_{\ue} \|_{\bL^2(\dist^{\alpha},T)} + \| e_{\pe} \|_{L^2(\dist^{\alpha},T)} \right) \| \nabla \boldsymbol{\Upsilon}_T \|_{\bL^2(\dist^{-\alpha},T)}.
\label{eq:estimate_XT}
\end{equation}
We now recall that $\boldsymbol{\Upsilon}_T:= \varphi_T \mathbf{X}_T$ and invoke estimate \eqref{eq:chi_bubble_T} to arrive at
\[
 \| \nabla \boldsymbol{\Upsilon}_T \|_{\bL^2(\dist^{-\alpha},T)} \lesssim h_T^{-1}D_T^{-\alpha/2}\| \mathbf{X}_T \|_{\bL^2(T)}.
\]
The previous two estimates yield the desired bound on the first term
\begin{multline}
h_T^{2}D_T^{\alpha} \|  \Delta \bu_{\T} -  (\ue_{\T} \cdot \nabla) \ue_{\T} - \DIV \ue_{\T} \ue_{\T} - \nabla \pe_{\T} \|_{\bL^2(T)}^2
\\
\lesssim \| \nabla \boldsymbol{e}_{\ue} \|^2_{\bL^2(\dist^{\alpha},T)} + \| e_{\pe} \|^2_{L^2(\dist^{\alpha},T)}.
\label{eq:estimate_interior_residual}
\end{multline}

\noindent \emph{Step 2}. In this step we control the jump term 
$h_T D_T^{\alpha} \| \llbracket(\GRAD \ue_\T - \pe_{\T} \mathbf{I}) \cdot \boldsymbol{\nu}\rrbracket \|_{\bL^2(S)}^{2}$.
 Let $T \in \T$ and $S \in \Sides_T$. Define $\boldsymbol{\Lambda}_S =   \varphi_S \llbracket (\GRAD \ue_\T - \pe_{\T} \mathbf{I}) \cdot \boldsymbol{\nu}\rrbracket$ so that, using \eqref{eq:bubble_S}, we have
\begin{equation}
\begin{aligned}
 \|  \llbracket (\nabla \bu_{\T}  - \pe_{\T} \mathbf{I}) \cdot \boldsymbol{\nu}\rrbracket  \|_{\bL^2(S)}^{2} & \lesssim \int_S |  \llbracket (\nabla \bu_{\T}  - \pe_{\T} \mathbf{I}) \cdot \boldsymbol{\nu}\rrbracket |^2 \varphi_S 
 \\
 & = \int_S  \llbracket (\nabla \bu_{\T}  - \pe_{\T} \mathbf{I}) \cdot \boldsymbol{\nu}\rrbracket \cdot \boldsymbol{\Lambda}_S.
 \end{aligned}
 \label{eq:estimate_jump_aux}
\end{equation}
Setting $\bv = \boldsymbol{\Lambda}_S$ in \eqref{eq:indentity_Phi_residual_new} and using that $\boldsymbol{\Lambda}_S(z) = \boldsymbol0$ yields
\[
 \int_S   \llbracket (\nabla \bu_{\T}  - \pe_{\T} \mathbf{I}) \cdot \boldsymbol{\nu}\rrbracket  \cdot \boldsymbol{\Lambda}_S =   a(\boldsymbol{\Phi},\boldsymbol{\Lambda}_S) - \sum_{T' \in \mathcal{N}_S} \int_{T'} \mathbf{X}_{T'} \cdot \boldsymbol{\Lambda}_S.
\]
Now, since $\supp (\boldsymbol{\Lambda}_S) \subset R_S:= \supp(\varphi_S) \subset T_{*} \cup T_{*}' \subset \cup\{T' : T' \in \mathcal{N}_{S}\}$,  similar arguments to the ones that led to \eqref{eq:bound_for_a_local} allow us to obtain
\begin{multline*}
 \int_S   \llbracket (\nabla \bu_{\T}  - \pe_{\T} \mathbf{I}) \cdot \boldsymbol{\nu}\rrbracket  \cdot \boldsymbol{\Lambda}_S \leq \\
 \sum_{T' \in \mathcal{N}_S} \| \nabla \boldsymbol{\Phi} \|_{\bL^2(\dist^{\alpha},T')} \| \nabla \boldsymbol{\Lambda}_S \|_{\bL^2(\dist^{-\alpha},T')} 
 + \sum_{T' \in \mathcal{N}_S} \| \mathbf{X}_{T'} \|_{\bL^2(\dist^{\alpha},T')} \| \boldsymbol{\Lambda}_S \|_{\bL^2(\dist^{-\alpha},T')}
\lesssim \\
\sum_{T' \in \mathcal{N}_S} 
\left( \| \nabla \boldsymbol{e}_{\ue} \|_{\bL^2(\dist^{\alpha},T')} + \| e_{\pe} \|_{L^2(\dist^{\alpha},T')} \right) 
\| \nabla \boldsymbol{\Lambda}_S \|_{\bL^2(\dist^{-\alpha},T')} \\ + \sum_{T' \in \mathcal{N}_S} \| \mathbf{X}_{T'} \|_{\bL^2(T')} \| \boldsymbol{\Lambda}_S \|_{\bL^2(T')}.
\end{multline*}
By shape regularity we have that
\[
  \| \boldsymbol{\Lambda}_S \|_{\bL^2(T')} \approx |T'|^{\frac{1}{2}} |S|^{-\frac{1}{2}} \| \boldsymbol{\Lambda}_S \|_{\bL^2(S)} \approx h_{T'}^{\frac{1}{2}} \| \boldsymbol{\Lambda}_S \|_{\bL^2(S)}.
\]
 This, estimate \eqref{eq:chi_bubble_S}, and the bound on $\mathbf{X}_{T'}$ derived in \eqref{eq:estimate_interior_residual} yield
\begin{multline*}
 \int_S   \llbracket (\nabla \bu_{\T}  - \pe_{\T} \mathbf{I}) \cdot \boldsymbol{\nu}\rrbracket  \cdot \boldsymbol{\Lambda}_S 
 \lesssim \\
 \sum_{T' \in \mathcal{N}_S} \left( \| \nabla \boldsymbol{e}_{\ue} \|_{\bL^2(\dist^{\alpha},T')} + \| e_{\pe} \|_{L^2(\dist^{\alpha},T')} \right)  
h_T^{-\frac{1}{2}} D_T^{-\frac{\alpha}{2}} \| \boldsymbol{\Lambda}_S \|_{\bL^2(S)}.
\end{multline*}
We replace the previous estimate in \eqref{eq:estimate_jump_aux} to arrive at
\begin{equation}
 h_T D_T^{\alpha}\|  \llbracket (\nabla \bu_{\T}  - p_{\T} \mathbf{I}) \cdot \boldsymbol{\nu}\rrbracket  \|_{\bL^2(S)}^{2} \lesssim \sum_{T' \in \mathcal{N}_S} \left( \| \nabla \boldsymbol{e}_{\ue} \|^2_{\bL^2(\dist^{\alpha},T')} + \| e_{\pe} \|^2_{L^2(\dist^{\alpha},T')} \right) ,
 \label{eq:estimate_jump}
\end{equation}
and since every $T \in \T$ belongs to at most two $\mathcal{N}_S$ for $S \in \Sides$, we can conclude.

\noindent \emph{Step 3}. We now bound the residual term associated with the incompressibility constraint. Since $\DIV \ue = 0$, for any $T \in \T$, we immediately arrive at
\begin{equation}
 \| \DIV \ue_{\T} \|_{L^2(\dist^{\alpha},T)} =  \| \DIV \boldsymbol{e}_{\ue}  \|_{L^2(\dist^{\alpha},T)} \lesssim \| \nabla \boldsymbol{e}_{\ue} \|_{\bL^2(\dist^{\alpha},T)}. 
 \label{eq:estimate_div}
\end{equation}

\noindent \emph{Step 4}. We now bound the term associated with the singular source. Let $T \in \T$ and note that, if $T \cap \{ z\} = \emptyset$, then there is nothing to prove.  Otherwise, we must obtain a bound for the term $h_{T}^{\alpha} | \bF |^2 $. To do so we follow the arguments developed in the proof of \cite[Theorem 5.3]{AGM} that yield the existence of a smooth function $\eta$ such that
\begin{equation}
 \eta(z) = 1,\quad \| \eta \|_{L^{\infty}(\Omega)} = 1, \quad \| \nabla \eta \|_{\bL^{\infty}(\Omega)} \lesssim h_T^{-1},
 \quad \supp(\eta) \subset \mathcal{N}_{T}.
\end{equation}
Define $\bv_{\eta}:= \bF \eta \in \bH_0^1(\dist^{-\alpha},\Omega)$. Since $(\ue,\pe)$ and $(\boldsymbol{\Phi},\psi)$ solve \eqref{eq:NSEVar} and \eqref{eq:Ritz}, respectively, we obtain
\begin{align*}
| \bF |^2 = \langle \bF \delta_z, \bv_{\eta} \rangle  & =
 a(\ue,\bv_{\eta}) + b_-(\bv_{\eta},\pe) + c(\ue,\ue;\bv_{\eta})
 \\
& =  a(\boldsymbol{\Phi}, \bv_{\eta}) + a(\ue_{\T},\bv_{\eta}) + b_-(\bv_{\eta},\pe_{\T}) + c(\ue_{\T},\ue_{\T};\bv_{\eta}).
\end{align*}
Since $\supp(\eta) \subset \mathcal{N}_{T}$, we apply similar arguments to the ones that led to \eqref{eq:bound_for_a_local}, integration by parts, and basic estimates to arrive at
\begin{align*}
| \bF |^2 
&  \lesssim \left( \| \GRAD \be_{\bu} \|_{\bL^2(\dist^{\alpha},\mathcal{N}_T)} + \| e_{\pe} \|_{L^2(\dist^{\alpha},\mathcal{N}_T)} \right) \| \GRAD \bv_{\eta} \|_{\bL^2(\dist^{-\alpha},\mathcal{N}_T)} 
\\
& + 
\sum_{T' \in \T: T' \subset \mathcal{N}_T} \| \Delta \bu_{\T} -  (\ue_{\T} \cdot \nabla) \ue_{\T} - \DIV \ue_{\T} \ue_{\T} - \nabla \pe_{\T}  \|_{\bL^2(T')} 
\| \bv_{\eta} \|_{\bL^2(T')} 
\\
& + 
\sum_{T' \in \T: T' \subset \mathcal{N}_T} \sum_{S\in \Sides_{T'}: S \not\subset \partial \mathcal{N}_T} \|  \llbracket (\nabla \bu_{\T}  - \pe_{\T} \mathbf{I}) \cdot \boldsymbol{\nu}\rrbracket  \|_{\bL^2(S)}  
\| \bv_{\eta} \|_{\bL^2(S)}.
\end{align*}
We now use the estimates 
\begin{equation*}
\| \eta\|_{L^2(S)} \lesssim h_T^{\frac{d-1}{2}},
\quad
\| \eta \|_{L^2(\mathcal{N}_T)} \lesssim  h_T^{\frac{d}{2}},
\quad
\| \GRAD \eta \|_{\bL^2(\dist^{-\alpha},\mathcal{N}_T)} \lesssim h_T^{\frac{d-2}{2}-\frac{\alpha}{2}},
\end{equation*}
and the fact that, since $z \in T$, we have $h_T \approx D_T$, to assert the bound
\begin{multline}
 | \bF |^2 \lesssim h_T^{\frac{d-2}{2}-\frac{\alpha}{2}} |\bF|  \left( \| \GRAD \be_{\bu} \|^2_{\bL^2(\dist^{\alpha},\mathcal{N}_T)} + \| e_{\pe} \|^2_{L^2(\dist^{\alpha},\mathcal{N}_T)} \right)^{\frac{1}{2}}
 \\
 + h_T^{\frac{d-2}{2}-\frac{\alpha}{2}} |\bF|  \Bigg ( \sum_{T' \in \T: T' \subset \mathcal{N}_T} h_{T'} D_{T'}^{\frac{\alpha}{2}}  \|\Delta \bu_{\T} -  (\ue_{\T} \cdot \nabla) \ue_{\T} - \DIV \ue_{\T} \ue_{\T} - \nabla \pe_{\T} \|_{\bL^2(T')}
 \\
 + \sum_{T' \in \T: T' \subset \mathcal{N}_{T'}} \sum_{S\in \Sides_{T'}: S \not\subset \partial \mathcal{N}_T} D_{T'}^{\frac{\alpha}{2}} h_{T'}^{\frac{1}{2}} \|  \llbracket (\nabla \bu_{\T}  - \pe_{\T} \mathbf{I}) \cdot \boldsymbol{\nu}\rrbracket  \|_{\bL^2(S)} \Bigg).
\end{multline}
Invoke \eqref{eq:estimate_interior_residual} and \eqref{eq:estimate_jump} to conclude.

\noindent \emph{Step 5}. Collect the estimates derived in the previous steps to arrive at the desired local efficiency estimate \eqref{eq:local_lower_bound}.
\end{proof}

\section{Numerical results}
\label{sec:numerics}
In this section we present a series of numerical examples that illustrate the performance of the devised error estimator $\E_{\alpha}(\bu_{\T},\pe_{\T};\T)$ defined in \eqref{eq:global_estimator}. In some
of these examples, we go beyond the presented theory and perform numerical experiments where we violate the assumption of homogeneous Dirichlet boundary conditions. The examples been carried out with the help of a code that we implemented using \texttt{C++}.  All matrices have been assembled exactly and global linear systems were solved using the multifrontal massively parallel sparse direct solver (MUMPS) \cite{MR1856597,MR2202663}. The right hand sides and terms involving the weight, and  the  approximation errors, are computed by a quadrature formula which is exact for polynomials of degree nineteen (19).

For a given partition $\T$, we solve the discrete problem \eqref{eq:NSEh} with the discrete spaces \eqref{eq:th_V}--\eqref{eq:th_P}. This setting will be referred to as Taylor--Hood approximation. To obtain the solution of \eqref{eq:NSEh} we use a fixed--point strategy, which is described in \textbf{Algorithm}~\ref{Algorithm-Fixed-Point}. In this algorithm, the initial guess is obtained as a discrete approximation of the solution to a Stokes problem with singular sources \cite{MR3892359} and $\textrm{tol}=10^{-8}$. Once the discrete solution $(\bu_{\T},\pe_{\T})$ is obtained, we compute, for $T \in \T$, the a posteriori error indicators $\E_{\alpha}(\bu_{\T},\pe_{\T};T)$, given in \eqref{eq:local_indicator}, to drive the adaptive mesh refinement procedure described in  \textbf{Algorithm}~\ref{Algorithm}. Every mesh $\T$ is adaptively refined by marking for refinement the elements $T \in \T$ that are such that the step 3 in \textbf{Algorithm}~\ref{Algorithm} holds. A sequence of adaptively refined meshes is thus generated from the initial meshes shown in Figure \ref{fig:meshes0}.

We define the total number of degrees of freedom as $\textsf{Ndof}:=\dim \mathbf{V}(\T)  + \dim \mathcal{P}(\T)$. We recall that the discrete spaces $\mathbf{V}(\T)$ and $\mathcal{P}(\T)$ are as in \eqref{eq:th_V} and \eqref{eq:th_P}, respectively. 

\begin{figure}
\centering
\begin{minipage}{0.16\linewidth}
\centering
\includegraphics[trim={0 0 0 0},clip,width=2cm,height=2cm,scale=1]{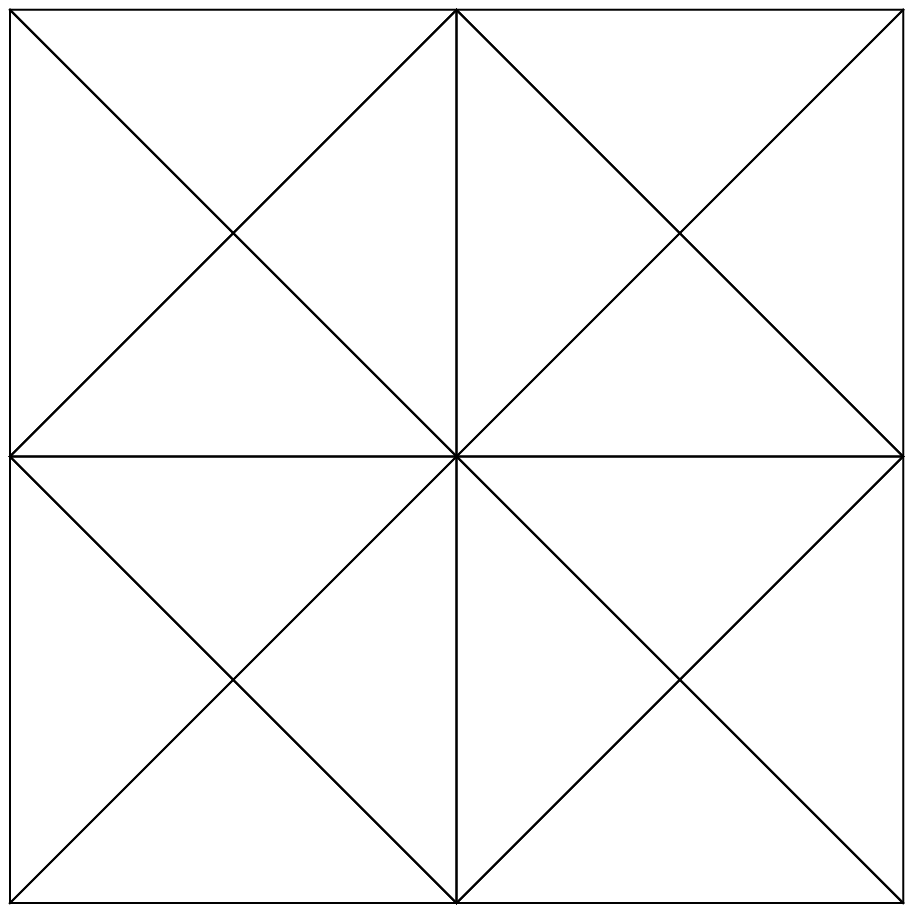}
\\
(a)
\end{minipage}
\begin{minipage}{0.16\linewidth}
\centering
\includegraphics[trim={0 0 0 0},clip,width=2cm,height=2cm,scale=1]{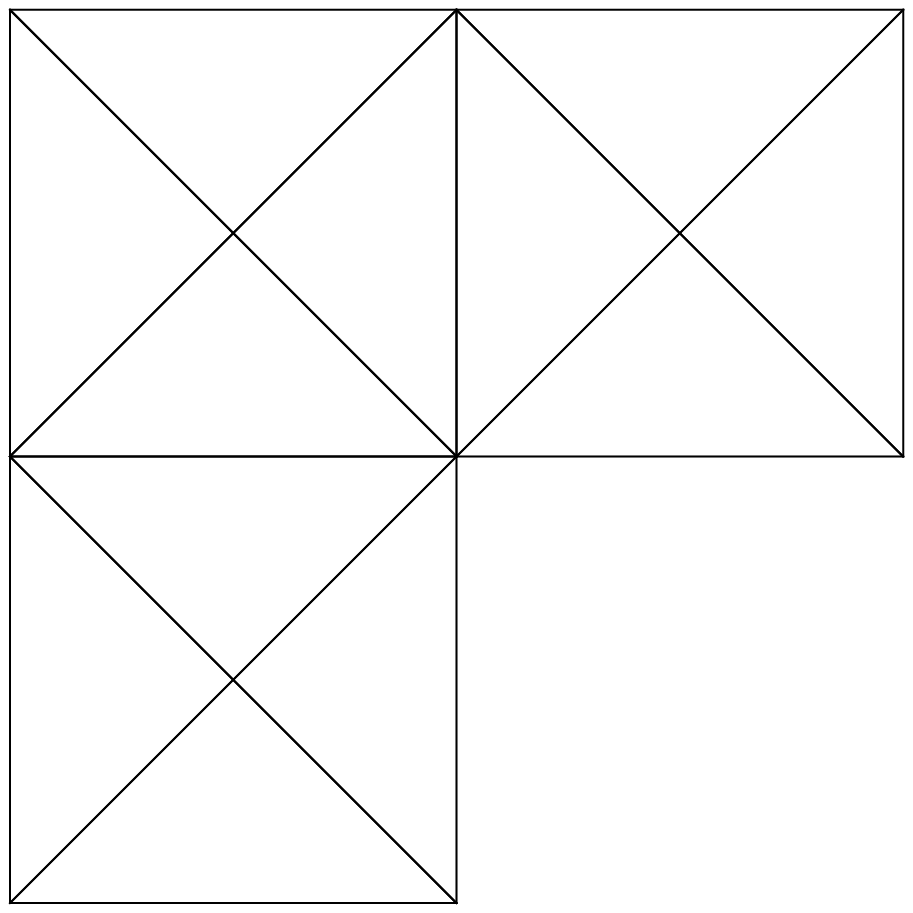}
\\
(b)
\end{minipage}
\begin{minipage}{0.3\linewidth}
\centering
\includegraphics[trim={0 0 0 0},clip,width=3.7cm,height=2cm,scale=1]{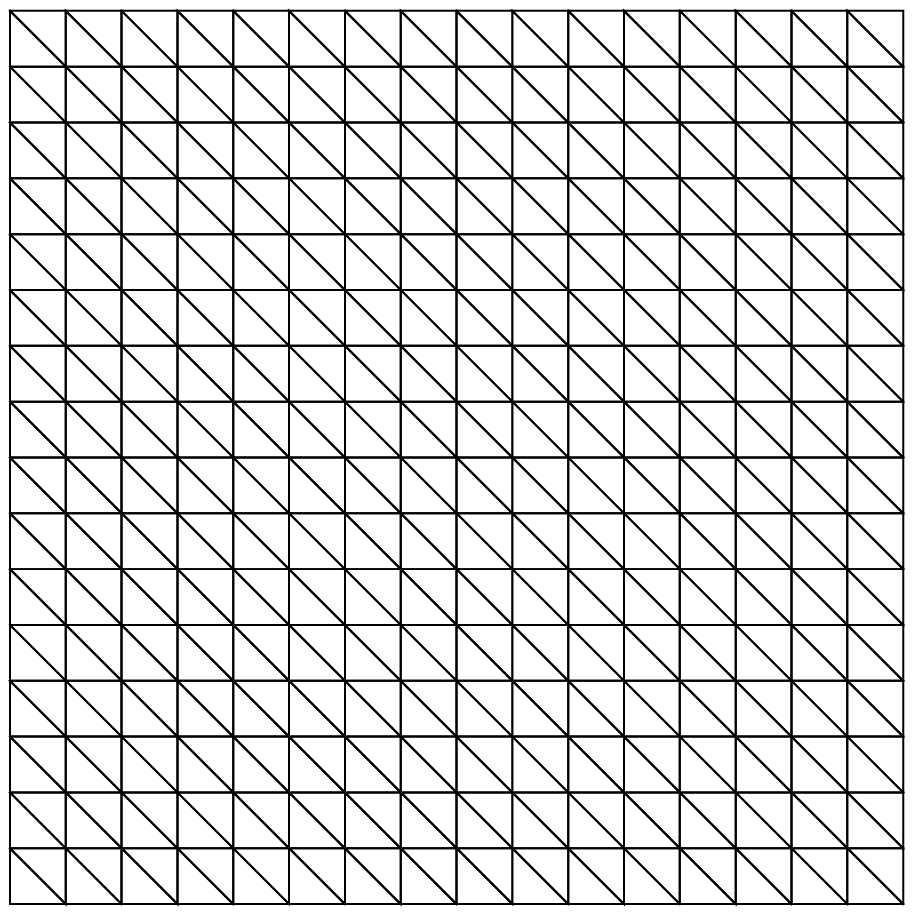}
\\
(c)
\end{minipage}
\begin{minipage}{0.33\linewidth}
\centering
\includegraphics[trim={0 0 0 0},clip,width=4cm,height=2.05cm,scale=1]{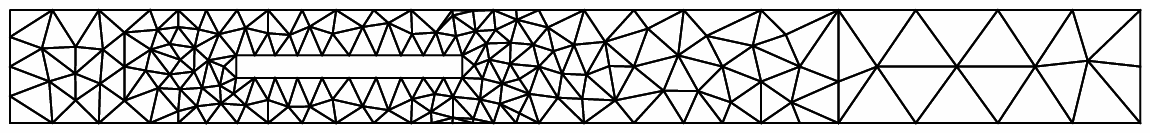}
\\
(d)
\end{minipage}
\caption{The initial meshes $\T_{0}$ used in the adaptive \textbf{Algorithm}~\ref{Algorithm} when \textrm{(a)} $\Omega=(0,1)^{2}$, \textrm{(b)} $\Omega=(-1,1)^{2}\setminus[0,1)\times[-1,0)$, \textrm{(c)} $\Omega=(0,4)\times(0,1)$, and \textrm{(d)} $\Omega=(0,10)^{2}\setminus[2,4]\times[0.4,0.6]$.}
\label{fig:meshes0}
\end{figure}

\footnotesize{
\begin{algorithm}[ht]
\caption{\textbf{Adaptive Algorithm}}
\label{Algorithm}
Input: Initial mesh $\T_{0}$, interior point $z\in\Omega$, and $\alpha\in(0,2)$;\\
\textbf{1:} Solve the discrete problem \eqref{eq:NSEh} by using \textbf{Algorithm} \ref{Algorithm-Fixed-Point}; \\  
\textbf{2:} For each $T \in \mathscr{T}$ compute the local error indicators $\E_\alpha(\bu_{\T},p_{\T};T)$ given in \eqref{eq:local_indicator};
\\
\textbf{3:} Mark an element $T \in \T$ for refinement if 
\vspace*{-0.3cm}
\[
  \E_\alpha(\bu_{\T},p_{\T};T) > \frac12 \max_{T' \in \T} \E_\alpha(\bu_{\T},p_{\T};T');
  \vspace*{-0.3cm}
\]
\textbf{4:} From step $\boldsymbol{3}$, construct a new mesh, using a longest edge bisection algorithm. Set $i \leftarrow i + 1$, and go to step $\boldsymbol{1}$.
\end{algorithm}}
%%%%%%%%%%%%%%%%%%%%%%%%%%%%%%%%%5
\footnotesize{
\begin{algorithm}[ht]
\caption{\textbf{Fixed-Point Algorithm}}
\label{Algorithm-Fixed-Point}
Input: Initial guess $(\bu_{\T}^{0},\pe_{\T}^{0})\in \mathbf{V}(\T)\times \mathcal{P}(\T)$ and $\textrm{tol}$. Set $i=1$;\\
\textbf{1:} Find $(\bu_{\T}^{i},p_{\T}^{i})\in \mathbf{V}(\T)\times \mathcal{P}(\T)$ such that \vspace*{-0.3cm}
\begin{equation*}
a(\bu_{\T}^{i},\bv_{\T}) + b_-(\bv_{\T}, \pe_{\T}^{i}) + c( \bu_{\T}^{i-1}, \bu_{\T}^{i};  \bv_{\T}) = \bF \cdot \bv_{\T}(z),
\quad
b_+(\bu_{\T}^{i},q_{\T}) = 0,\vspace*{-0.2cm}
\end{equation*}
for all $(\bv_{\T},q_{\T})\in \mathbf{V}(\T)\times \mathcal{P}(\T)$;
\\
\textbf{2:} If $|(\bu_{\T}^{i},p_{\T}^{i})-(\bu_{\T}^{i-1},p_{\T}^{i-1})|>\textrm
{tol}$, set $i \leftarrow i + 1$, and go to step $\boldsymbol{1}$. Otherwise, \textbf{return} $(\bu_{\T},\pe_{\T})=(\bu_{\T}^{i},\pe_{\T}^{i})$.
\end{algorithm}}
\normalsize

\subsection{Convex and non--convex domains with homogeneous boundary conditions} 
We first explore the performance of the devised a posteriori error estimator in problems  with homogeneous boundary conditions on convex and non-convex domains $\Omega$.

\subsubsection{Convex domain} We set $\Omega=(0,1)^{2}$, $z=(0.5,0.5)^\intercal$, and $\bF=(1,1)^\intercal$. We explore the performance of $\E_{\alpha}$ when driving
the adaptive procedure of \textbf{Algorithm} \ref{Algorithm}. We also investigate the effect of varying
the exponent $\alpha$ in the Muckenhoupt weight. To accomplish this task, we consider $\alpha\in\{0.25,0.75,1.0,1.25,1.25,1.5,1.75\}$. 

In Figure \ref{fig:convergence_alpha} we present the experimental rates of convergence for the error estimator $\E_{\alpha}$. We observe that optimal experimental rates of convergence are attained for all the values of the parameter $\alpha$ that we considered. We also observe that most of the refinement is concentrated around the singular source point. 

\begin{figure}
\begin{minipage}[c]{.5\textwidth}
\centering
\psfrag{estimador 0.25}{$\E_{0.25}(\bu_{\T},\pe_{\T};\T)$}
\psfrag{estimador 0.75}{$\E_{0.75}(\bu_{\T},\pe_{\T};\T)$}
\psfrag{estimador 1.0}{$\E_{1.0}(\bu_{\T},\pe_{\T};\T)$}
\psfrag{estimador 1.25}{$\E_{1.25}(\bu_{\T},\pe_{\T};\T)$}
\psfrag{estimador 1.5}{$\E_{1.5}(\bu_{\T},\pe_{\T};\T)$}
\psfrag{estimador 1.75}{$\E_{1.75}(\bu_{\T},\pe_{\T};\T)$}
\psfrag{rate(h2)}{$\textrm{Ndof}^{-1}$}
\centering
\includegraphics[trim={0 0 0 0},clip,width=6.5cm,height=5.5cm,scale=0.8]{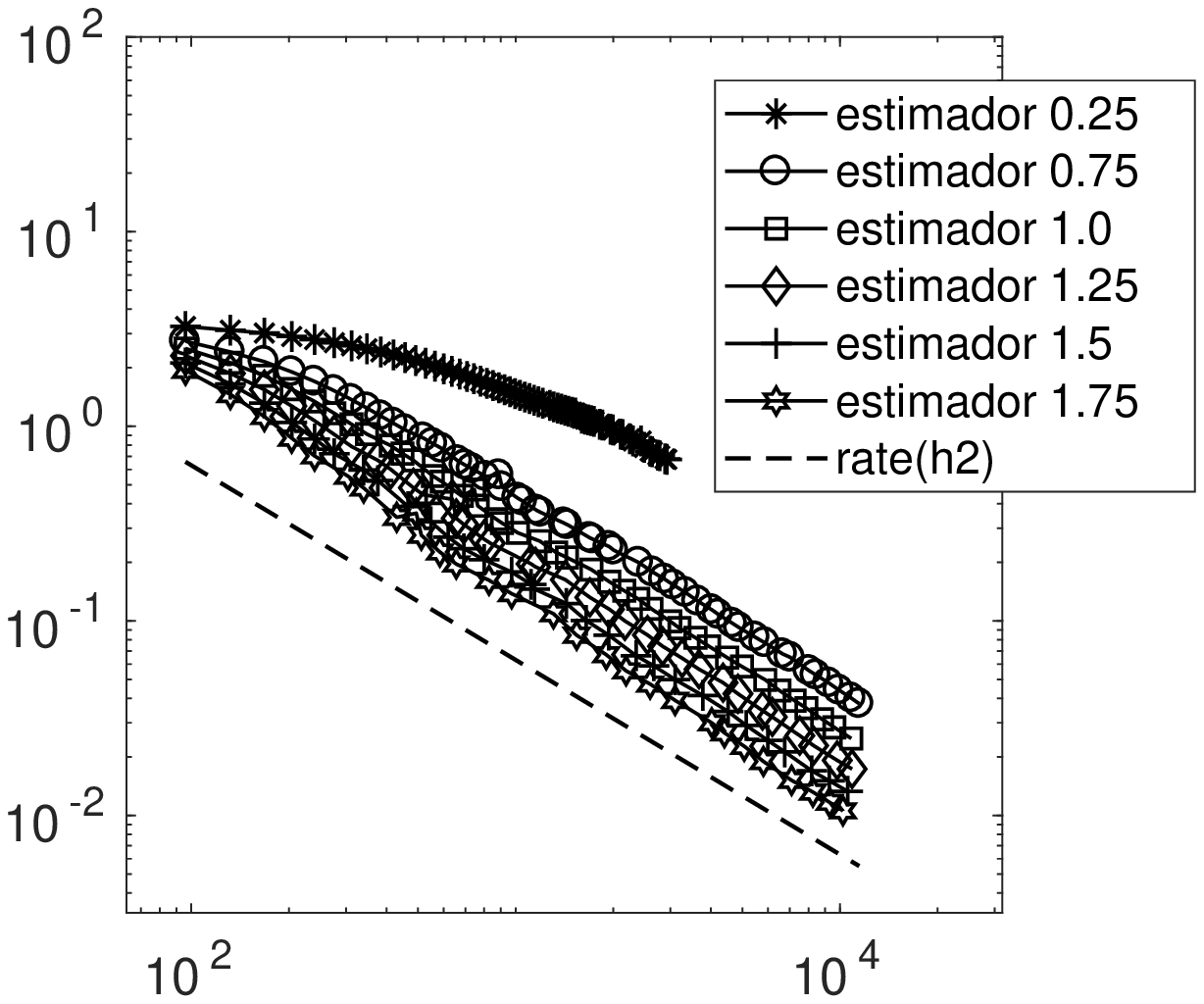}
\\
$\textrm{Ndof}$
\end{minipage}
\hfill
\begin{minipage}[c]{.4\textwidth}
\centering
\includegraphics[trim={0 0 0 0},clip,width=5cm,height=5cm,scale=0.8]{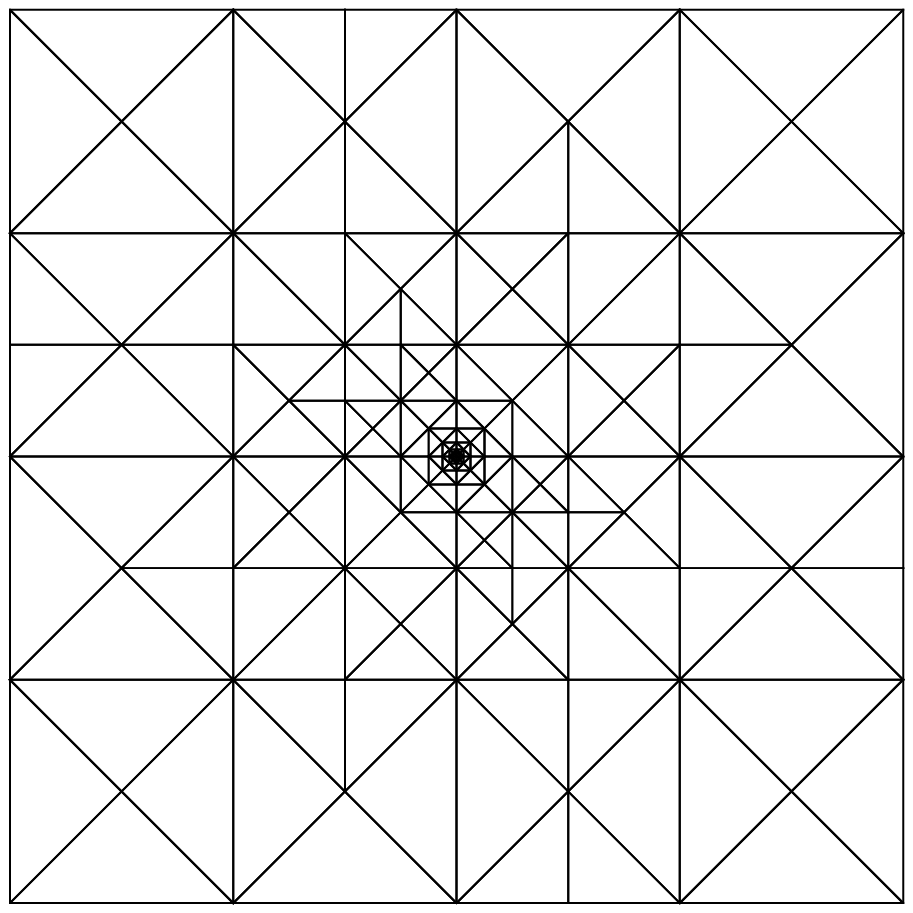}
\\
$~$
\end{minipage}
\caption{\textbf{Convex domain:} Experimental rates of convergence for the error estimator $\E_{\alpha}(\bu_{\T},\pe_{\T};\T)$ considering $\alpha\in\{0.25.0.5,1.0,1.25,1.5,1.75\}$ (left) and the mesh obtained after 15 adaptive refinements with $\alpha=1.5$ (right). The mesh contains 204 elements and 113 vertices.}
\label{fig:convergence_alpha}
\end{figure}

%%%%%%%%%%%%%%%%%%%%%%%%%%%

\subsubsection{Non-convex domain} We set $\Omega=(-1,1)^{2}\setminus[0,1)\times[-1,0)$, an L--shaped domain, $z=(0.5,0.5)^\intercal$, and $\bF=(1,1)^\intercal$. We again consider different values of the exponent $\alpha$ of the Muckenhoupt weight $\dist^{\alpha}$ defined in \eqref{distance_A2}. We consider $\alpha\in\{0.25,0.75,1.0,1.25,1.25,1.5,1.75\}$. 

In Figure \ref{fig:convergence_alpha_L} we present the experimental rates of convergence for the error estimator $\E_{\alpha}$ and the mesh obtained after 30 adaptive refinements with $\alpha = 1.5$. We observe that, for all the values of the parameter $\alpha$ that we have considered, optimal experimental rates of convergence are attained. We also observe that most of the refinement is concentrated around the singular source point and the reentrant corner. 

\begin{figure}
\begin{minipage}[c]{.5\textwidth}
\centering
\psfrag{estimador 0.25}{$\E_{0.25}(\bu_{\T},\pe_{\T};\T)$}
\psfrag{estimador 0.75}{$\E_{0.75}(\bu_{\T},\pe_{\T};\T)$}
\psfrag{estimador 1.0}{$\E_{1.0}(\bu_{\T},\pe_{\T};\T)$}
\psfrag{estimador 1.25}{$\E_{1.25}(\bu_{\T},\pe_{\T};\T)$}
\psfrag{estimador 1.5}{$\E_{1.5}(\bu_{\T},\pe_{\T};\T)$}
\psfrag{estimador 1.75}{$\E_{1.75}(\bu_{\T},\pe_{\T};\T)$}
\psfrag{rate(h2)}{$\textrm{Ndof}^{-1}$}
\centering
\includegraphics[trim={0 0 0 0},clip,width=5.5cm,height=5.5cm,scale=0.8]{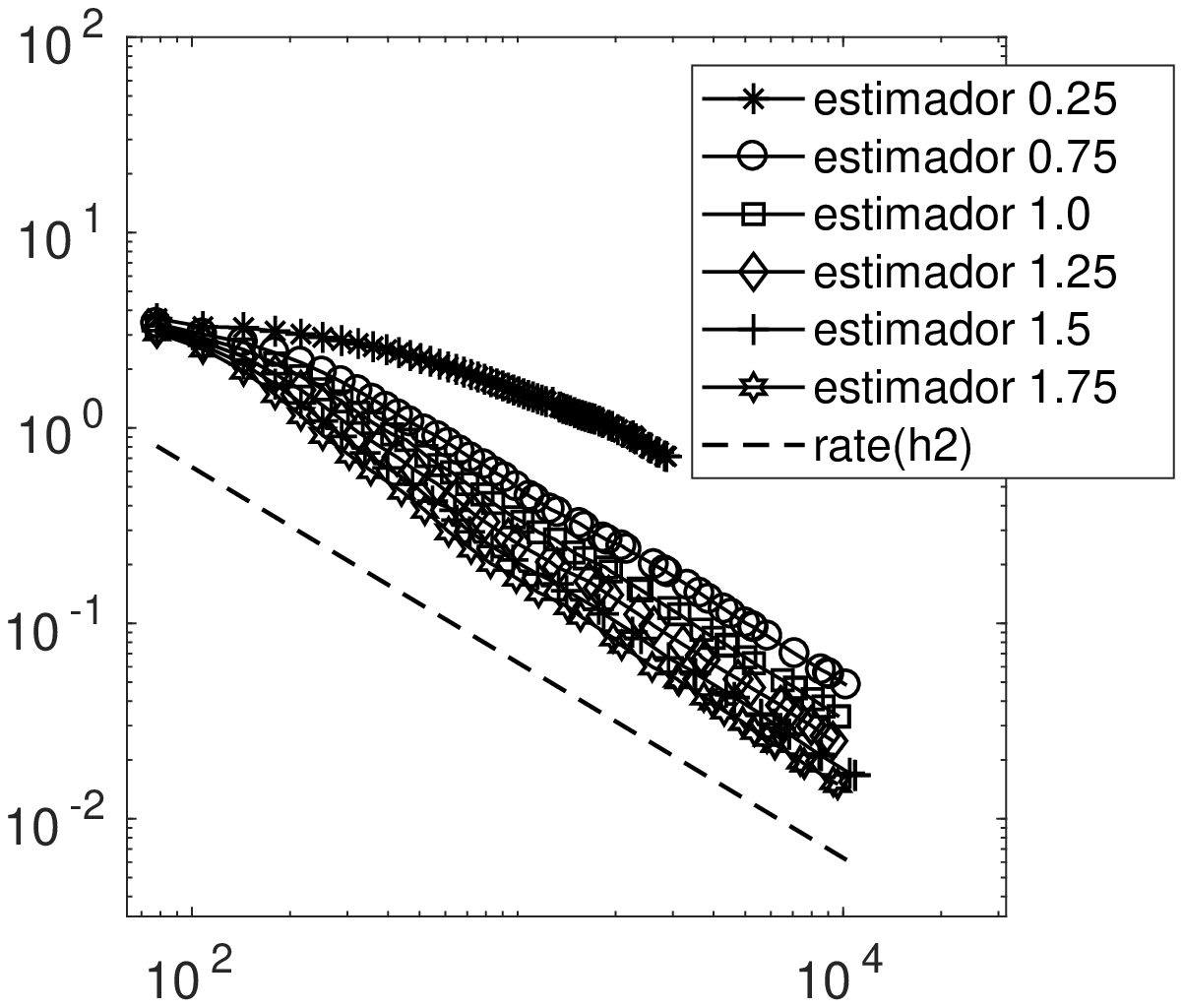}
\\
$\textrm{Ndof}$
\end{minipage}
\hfill
\begin{minipage}[c]{.4\textwidth}
\centering
\includegraphics[trim={0 0 0 0},clip,width=5cm,height=5cm,scale=0.8]{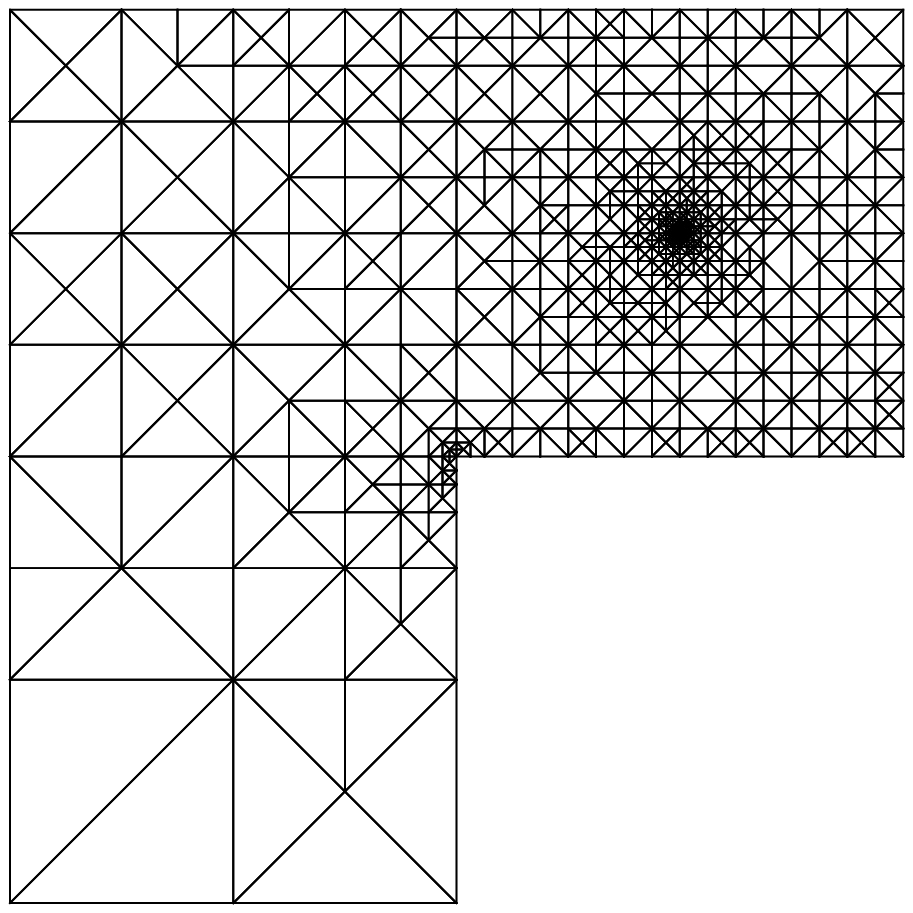}
\\
$~$
\end{minipage}
\caption{\textbf{Non-convex domain:} Experimental rates of convergence for the error estimator $\E_{\alpha}(\bu_{\T},\pe_{\T};\T)$ considering $\alpha\in\{0.25.0.5,1.0,1.25,1.5,1.75\}$ (left) and the mesh obtained after 30 adaptive refinements with $\alpha=1.5$ (right). The mesh contains 1208 elements and 639 vertices.}
\label{fig:convergence_alpha_L}
\end{figure}

%%%%%%%%%%%%%%%%%%%%%%%%%%%%
%%%%%%%%%%%%%%%%%%%%%%%%%%%%

\subsection{A series of Dirac sources}

We now go beyond the presented theory and include a series of Dirac delta sources on the right--hand side of the momentum equation. To be precise, we will replace the momentum equation in \eqref{eq:NSEStrong} by
\begin{equation}\label{eq:NSEStrongSS}
-\Delta \ue + (\ue \cdot \GRAD) \ue + \GRAD \pe = \sum_{z\in\mathcal{Z}}\bF_{z}\delta_{z} \text{ in } \Omega, 
\end{equation} 
where $\mathcal{Z}\subset\Omega$ denotes a finite set with cardinality $\#\mathcal{Z}$ which is such that $1< \#\mathcal{Z}$ and $\{\bF_z\}_{z \in \mathcal{Z}} \subset \R^2$. We introduce the weight \cite[Section 5]{MR3679932}
\begin{equation}\label{new_A2}
 \rho(x)=\left\{
 \begin{array}{lc}
 \dist^{\alpha}, & \exists~z\in\mathcal{Z}:|x-z|<\frac{d_{\mathcal{Z}}}{2} ,\\
 1, & |x-z|\geq\frac{d_{\mathcal{Z}}}{2},~\forall~z\in\mathcal{Z},
 \end{array}
 \right.
 \end{equation}
 where
 $ 
 d_{\mathcal{Z}}=
 \min
 \left\{
 \textsf{dist}(\mathcal{Z},\partial\Omega),\min\left\{|z-z'|:z,z'\in\mathcal{Z},z\neq z'\right\}
 \right\}
 $.
This weight belongs to the Muckenhoupt class $A_2$ \cite{MR3215609} and to the restricted class $A_2(\Omega)$. With the weight $\rho$ at hand, we modify the definition \eqref{XandY} of the spaces $\mathcal{X}$ and $\mathcal{Y}$ as follows:
\begin{equation}
 \label{new_XandY}
 \mathcal{X} = \bH^{1}_0(\rho,\Omega) \times L^2(\rho,\Omega)/ \mathbb{R}, \quad
 \mathcal{Y} = \bH^{1}_0(\rho^{-1},\Omega) \times L^2(\rho^{-1},\Omega)/ \mathbb{R}.
 \end{equation}
 
Define 
 \begin{equation}
\label{eq:DT_new}
  D_{T,\mathcal{Z}} := \min_{z \in \mathcal{Z}} \left\{  \max_{x \in T} |x-z| \right\}.
\end{equation}

 We propose the following error estimator when the Taylor--Hood scheme is considered:
$$
\mathscr{D}_{\alpha}(\bu_{\T},p_{\T};\T):=\left(\sum_{T\in\T}\mathscr{D}_{\alpha}^{2}(\bu_{\T},p_{\T};T)\right)^{\frac{1}{2}},
$$
where the local indicators are such that
 \begin{multline}
\mathscr{D}_{\alpha}(\bu_{\T},p_{\T};T):= \bigg( h_T^2D_{T,\mathcal{Z}}^{\alpha}   \|    \Delta \bu_{\T} -  (\ue_{\T} \cdot \nabla) \ue_{\T} - \DIV \ue_{\T} \ue_{\T} - \nabla \pe_{\T} \|_{\bL^2(T)}^2 
\\
+  \|  \DIV \bu_{\T} \|_{L^2(\rho,T)}^2 
+ h_T D_{T,\mathcal{Z}}^{\alpha}\| \llbracket (\nabla \bu_{\T}  - \pe_{\T} \mathbf{I})\rrbracket \cdot \boldsymbol{\nu} \|_{\bL^2(\partial T \setminus \partial \Omega)}^2 + \sum_{z \in \mathcal{Z}\cap T} h_{T}^{\alpha} | \bF_{z} |^2 \bigg)^{\frac{1}{2}}.
\label{eq:local_indicator_s}
\end{multline} 

\subsubsection{Convex domain with four Delta sources} We set $\Omega=(0,1)^{2}$ and
$$
\mathcal{Z}=\{(0.25,0.25),(0.25,0.75),(0.75,0.25),(0.75,0.75)\}.
$$
We consider $\bF_{z}=(1,1)^\intercal$ for all $z \in \mathcal{Z}$ and fix the exponent of the Muckenhoupt weight $\rho$ as $\alpha=1.5$.

In Figure \ref{fig:4points} we present the experimental rates of convergence for the error estimator $\mathscr{D}_{\alpha}$ and the mesh obtained after 30 adaptive refinements with $\alpha = 1.5$. It can be observed that the devised a posteriori error estimator exhibits an optimal experimental rate of convergence. It can also be observed that most of the refinement is concentrated around the singular source points. 
In Figure \ref{fig:4pointFEM}, we present the finite element approximations of $|\bu_{\T}|$ and $\pe_{\T}$ over the mesh that is obtained after 30 iterations of our adaptive loop with $\alpha=1.5$.

\begin{figure}
\centering
\psfrag{estimador 1.5}{$\mathscr{D}_{1.5}(\bu_{\T},\pe_{\T};\T)$}
\psfrag{rate(h2)}{$\textrm{Ndof}^{-1}$}
\includegraphics[trim={0 0 0 0},clip,width=5cm,height=5cm,scale=0.8]{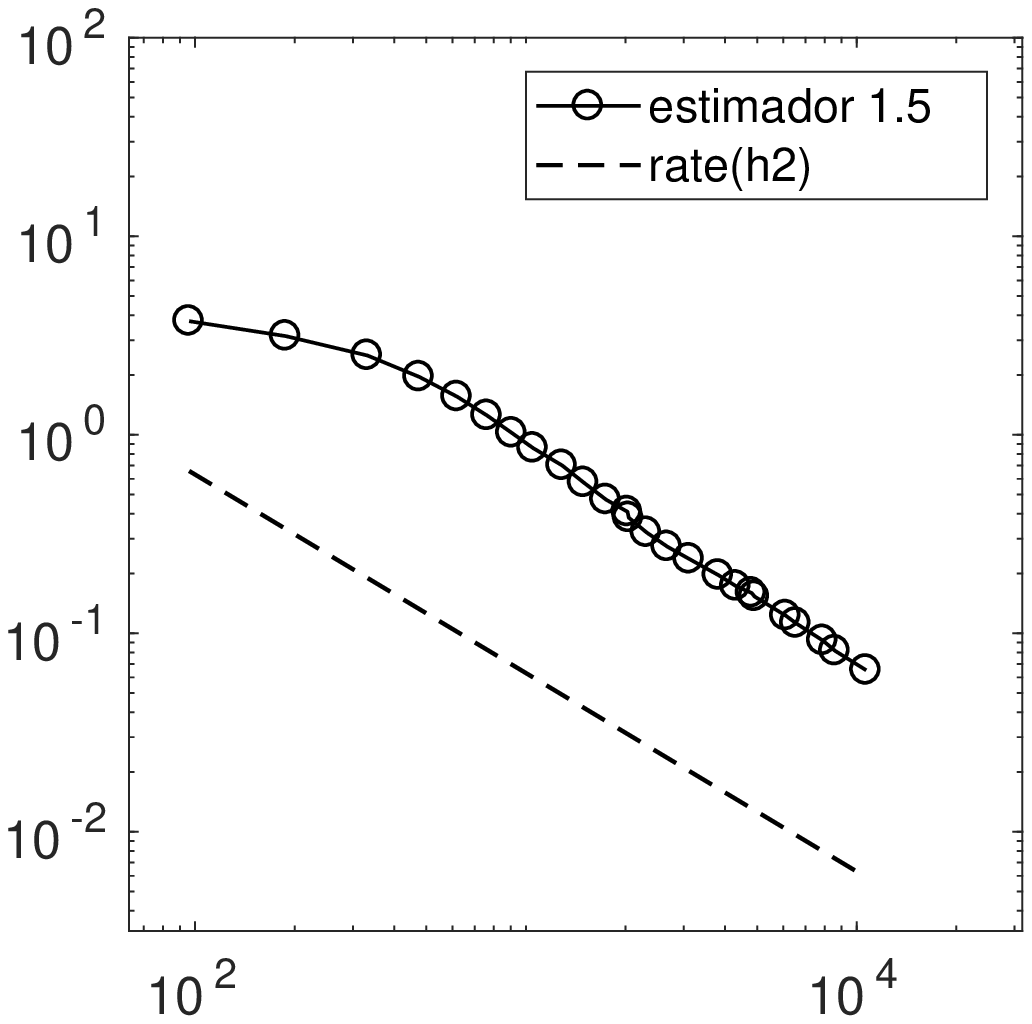}
\quad
\includegraphics[trim={0 0 0 0},clip,width=5cm,height=5cm,scale=0.8]{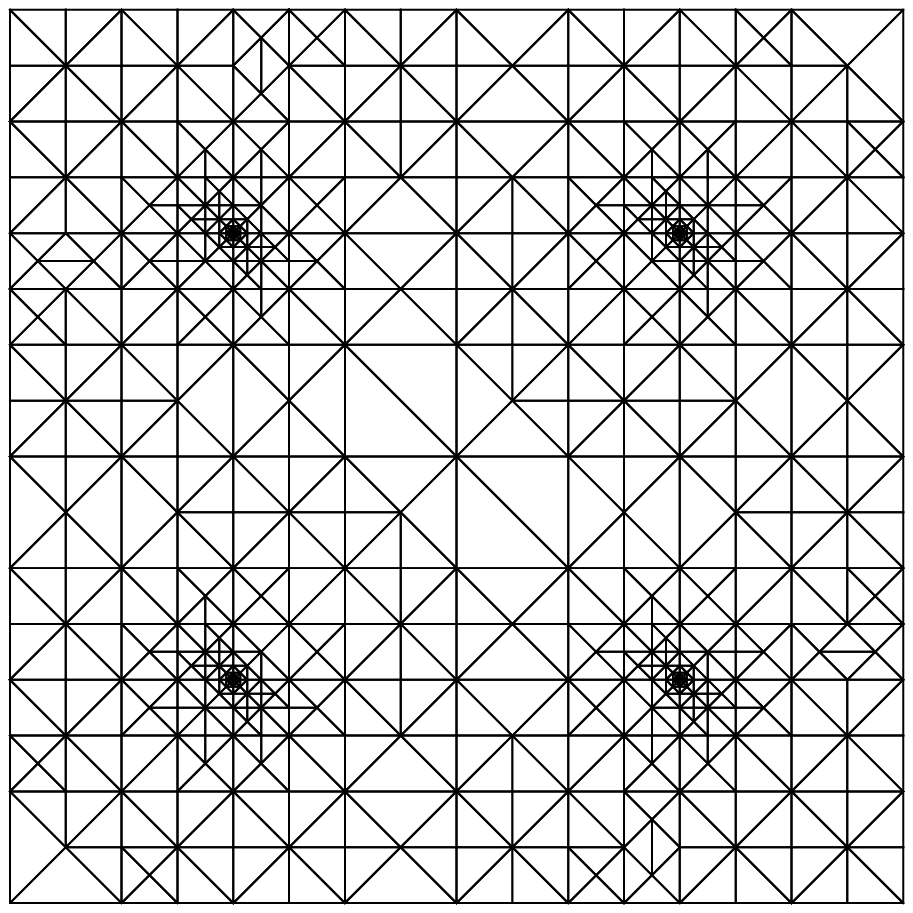}
\caption{\textbf{A series of Dirac sources:} Experimental rate of convergence for the error estimator $\mathscr{D}_{1.5}(\bu_{\T},\pe_{\T};\T)$ considering $\alpha=1.5$ (left) and the mesh obtained after 30 adaptive refinements (right). The mesh contains 1324 elements and 693 vertices.}
\label{fig:4points}
\end{figure}

\begin{figure}
\centering
\includegraphics[trim={0 0 0 0},clip,width=10cm,height=7.5cm,scale=0.8]{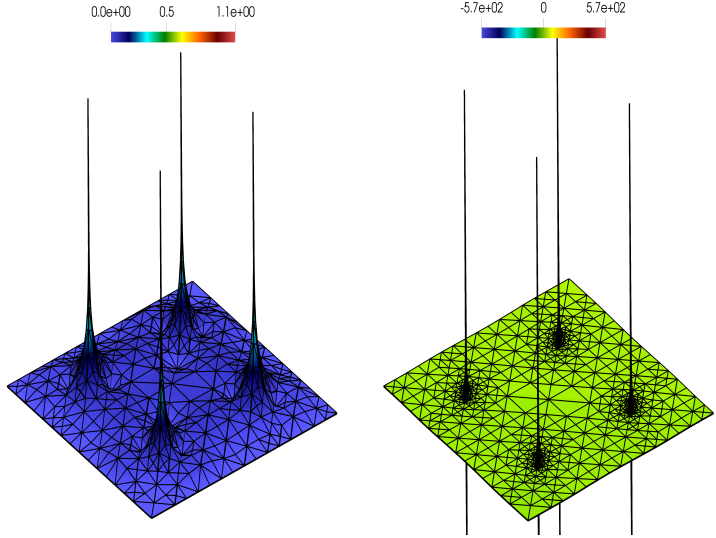}
\caption{\textbf{A series of Dirac sources:} Finite element approximations $|\bu_{\T}|$ (left) and $\pe_{\T}$ (right) over the mesh obtained after $30$ adaptive refinements with $\alpha = 1.5$. The mesh contains 1324 elements and 693 vertices.}
\label{fig:4pointFEM}
\end{figure}

%%%%%%%%%%%%%%%%%%%%%%%%%%%%
%%%%%%%%%%%%%%%%%%%%%%%%%%%%

\subsection{A convex domain with nonhomogeneous boundary conditions} 

We now explore the performance of our devised a posteriori error estimator by considering a problem with nonhomogeneous boundary conditions; a framework that does not fit in our analysis.

\subsubsection{A rectangular domain} We set $\Omega=(0,4)\times(0,1)$, $z=(0.5,0.5)^\intercal$ and $\bF=(10,10)^\intercal$. The boundary conditions are illustrated in the left panel of Figure \ref{fig:boundary_conditions}. We prescribe the parabolic Dirichlet inflow condition $\bu_{D}=(y(1-y),0)^\intercal$ on $\{ 0\} \times [0,1]$ and $\bu_{D} = \boldsymbol{0}$ on $[0,4] \times \{ 0 \} \cup  \{ 1 \}$ and the homogeneous Neumann condition $(\nabla\bu_{D}-p\mathbf{I})\boldsymbol{\nu}=\boldsymbol{0}$ on $\{4\} \times [0,1] $. Here,  $\boldsymbol{\nu}$ denotes the unit normal on $\partial \Omega$ pointing outwards. We recall that $\mathbf{I}$ denotes the identity matrix in $\mathbb{R}^{2 \times 2}$.

In the right panel of Figure \ref{fig:boundary_conditions} we observe an optimal decay rate for the devised error estimator $\E_{1.5}(\bu_{\T},\pe_{\T};\T)$. Finally, Figure~\ref{fig:nonhomogeneous_plot} shows the streamlines and velocity field for this problem.

\begin{figure}
  \begin{tabular}{lr}
    \begin{tikzpicture}[line cap=round,line join=round,>=triangle 45,x=0.5cm,y=0.5cm]
    \clip(-0.27,-1.29) rectangle (15,7);
    \draw (2,2)-- (12,2);
    \draw (12,2)-- (12,6);
    \draw (12,6)-- (2,6);
    \draw (2,6)-- (2,2);
    \draw (0.1,7) node[anchor=north west] {\rotatebox{90}{$\bu_{D}=(y(1-y),0)^\intercal$}};
    \draw (12.5,7) node[anchor=north west] {\rotatebox{90}{$(\nabla\bu_{D}-p\boldsymbol{I})\boldsymbol{\nu}=\boldsymbol{0}$}};
    \draw (5.5,2) node[anchor=north west] {$\bu_{D}=\boldsymbol{0}$};
    \draw (5.5,7) node[anchor=north west] {$\bu_{D}=\boldsymbol{0}$};
    \begin{scriptsize}
    \fill [color=black] (2,2) circle (1.5pt);
    \draw[color=black] (2,1.5) node {$(0,0)$};
    \fill [color=black] (12,2) circle (1.5pt);
    \draw[color=black] (12,1.5) node {$(4,0)$};
    \fill [color=black] (12,6) circle (1.5pt);
    \draw[color=black] (12,6.5) node {$(4,1)$};
    \fill [color=black] (2,6) circle (1.5pt);
    \draw[color=black] (2,6.5) node {$(0,1)$};
    \fill [color=black] (3,4) circle (1.5pt);
    \draw[color=black] (5,4) node {$z=(0.5,0.5)$};
    \end{scriptsize}
    \end{tikzpicture}
  &
    \psfrag{estimador 0.25}{$\E_{1.5}(\bu_{\T},\pe_{\T};\T)$}
    \psfrag{rate(h2)}{$\textrm{Ndof}^{-1}$}
    \includegraphics[trim={0 0 0 0},clip,width=4.5cm,height=4.5cm,scale=0.6]{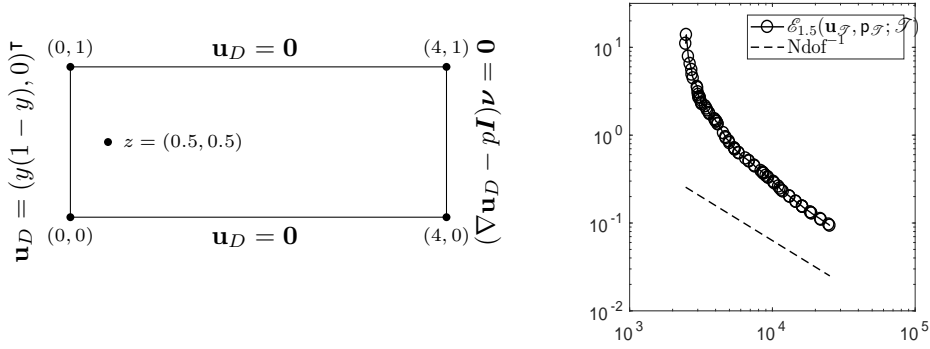}    
  \end{tabular}
\caption{\textbf{A rectangular domain:} Boundary conditions for problem \eqref{eq:NSEStrong} on $\Omega = (0,4) \times (0,1)$. We prescribe the parabolic Dirichlet inflow condition $\bu_{D}=(y(1-y),0)^\intercal$ on $\{ 0\} \times [0,1]$ and $\bu_{D} = \boldsymbol{0}$ on $[0,4] \times \{ 0 \} \cup  \{ 1 \}$ and the homogeneous Neumann condition $(\nabla\bu_{D}-p\mathbf{I})\boldsymbol{\nu}=\boldsymbol{0}$ on $\{4\} \times [0,1]$ (left). Experimental rate of convergence for the error estimator $\E_{\alpha}(\bu_{\T},\pe_{\T};\T)$ with $\alpha=1.5$ (right).}
\label{fig:boundary_conditions}
\end{figure}

\begin{figure}
\centering
\includegraphics[trim={0 0 0 0},clip,width=13cm,height=8cm,scale=0.8]{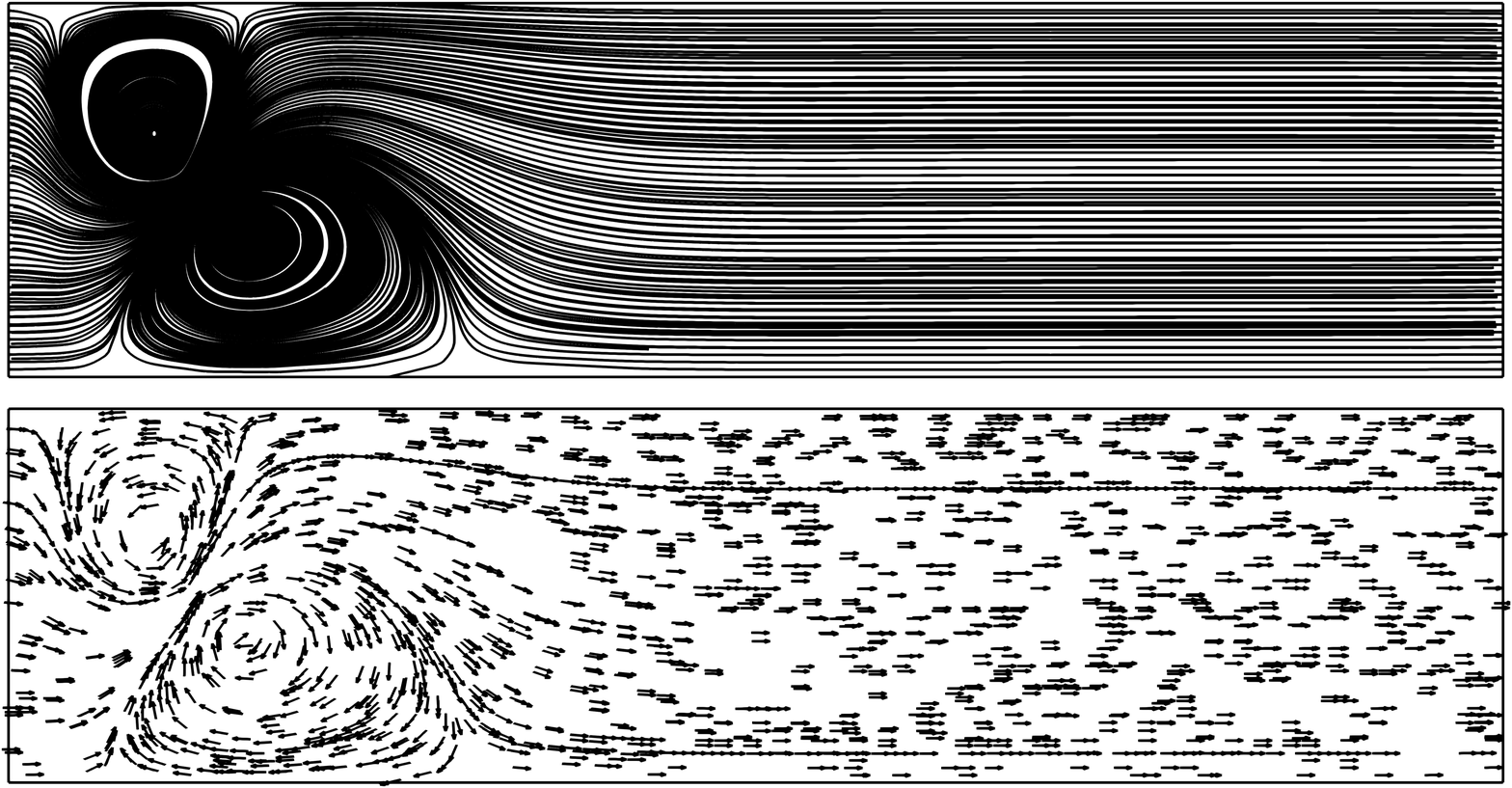}
\caption{\textbf{A rectangular domain:} Streamlines (top) and vector plot (bottom) associated to the velocity field $\bu_{\T}$ over the mesh, with 1485 elements and 781 vertices, that is obtained after 30 adaptive refinements. We have considered $\alpha=1.5$.}
\label{fig:nonhomogeneous_plot}
\end{figure}

\subsubsection{A rectangular domain with and obstacle} We set $\Omega = (0,10) \times (0,1)\setminus[2,4]\times[0.4,0.6]$ and $\bF_{z}=(10,10)^\intercal$ for $z\in\mathcal{Z}$, where $\mathcal{Z}=\{z_{i}\}_{i=1}^{4}$, with 
\begin{multline*}
z_{1}=(1.011635,0.198805)^\intercal,\quad
z_{2}=(1.011635,0.801195)^\intercal,
\\
z_{3}=(5.354725,0.200869)^\intercal,\quad
z_{4}=(5.264444,0.719518)^\intercal.
\end{multline*}
The boundary conditions are illustrated in the top panel of Figure \ref{fig:boundary_conditions_4points}. We prescribe the parabolic Dirichlet inflow condition $\bu_{D}=(y(1-y),0)^\intercal$ on $\{ 0\} \times [0,1]$ and $\bu_{D} = \boldsymbol{0}$ on $[0,10] \times \{ 0 \} \cup  \{ 1 \}$ and $\bu_{D} = \boldsymbol{0}$ on the boundary of $[4,6] \times [0.4,0.6]$, and the homogeneous Neumann condition $(\nabla\bu_{D}-p\mathbf{I})\boldsymbol{\nu}=\boldsymbol{0}$ on $\{10\} \times [0,1] $. Here,  $\boldsymbol{\nu}$ denotes the unit normal on $\partial \Omega$ pointing outwards. We recall again that $\mathbf{I}$ denotes the identity matrix in $\mathbb{R}^{2 \times 2}$. The bottom panel of Figure~\ref{fig:boundary_conditions_4points} shows the experimental rate of convergence for the error estimator $\mathscr{D}_{\alpha}(\bu_{\T},\pe_{\T};\T)$ with $\alpha=1.5$. The estimator exhibits an optimal rate of decay. Finally, Figure \ref{fig:final_obstacle} shows the streamlines, magnitude and vector plot of the velocity field and the mesh obtained after 100 iterations of our adaptive loop.

\begin{figure}
  \begin{tabular}{cc}
\begin{tikzpicture}[line cap=round,line join=round,>=triangle 45,x=1.0cm,y=1.4cm]
\clip(-0.8,-0.73) rectangle (10.76,1.83);
\draw (0,0)-- (10,0);
\draw (10,0)-- (10,1);
\draw (10,1)-- (0,1);
\draw (0,1)-- (0,0);
\draw (2,0.4)-- (4,0.4);
\draw (4,0.4)-- (4,0.6);
\draw (4,0.6)-- (2,0.6);
\draw (2,0.6)-- (2,0.4);
\draw (-0.8,1.5) node[anchor=north west] {\rotatebox{90}{$\bu_{D}=(y(1-y),0)^\intercal$}};
\draw (4.2,-0.08) node[anchor=north west] {$\bu_{D}=\boldsymbol{0}$};
\draw (4.2,1.4) node[anchor=north west] {$\bu_{D}=\boldsymbol{0}$};
\draw (10.28,1.4) node[anchor=north west] {\rotatebox{90}{$(\nabla\bu_{D}-p\boldsymbol{I})\boldsymbol{\nu}=\boldsymbol{0}$}};
\draw (2.3,0.68) node[anchor=north west] {$\bu_{D}=\boldsymbol{0}$};
\begin{scriptsize}
\fill [color=black] (0,0) circle (1.5pt);
\draw[color=black] (0.0,-0.2) node {$(0,0)$};
\fill [color=black] (10,0) circle (1.5pt);
\draw[color=black] (10,-0.2) node {$(10,0)$};
\fill [color=black] (10,1) circle (1.5pt);
\draw[color=black] (10,1.1) node {$(10,1)$};
\fill [color=black] (0,1) circle (1.5pt);
\draw[color=black] (0.0,1.1) node {$(0,1)$};
\fill [color=black] (2,0.4) circle (1.5pt);
\draw[color=black] (2,0.2) node {$(2,0.4)$};
\fill [color=black] (2,0.6) circle (1.5pt);
\draw[color=black] (2.05,0.69) node {};
\fill [color=black] (4,0.4) circle (1.5pt);
\draw[color=black] (4.06,0.49) node {};
\fill [color=black] (4,0.6) circle (1.5pt);
\draw[color=black] (4.06,0.69) node {$(4,0.6)$};
\fill [color=black] (0.82,0.19) circle (1.5pt);
\draw[color=black] (0.88,0.27) node {$z_{1}$};
\fill [color=black] (0.82,0.78) circle (1.5pt);
\draw[color=black] (0.87,0.65) node {$z_{2}$};
\fill [color=black] (5.28,0.19) circle (1.5pt);
\draw[color=black] (5.34,0.27) node {$z_{3}$};
\fill [color=black] (5.14,0.77) circle (1.5pt);
\draw[color=black] (5.19,0.65) node {$z_{4}$};
\end{scriptsize}
\end{tikzpicture}
  \\
    \psfrag{estimador 1.5}{$\mathscr{D}_{1.5}(\bu_{\T},\pe_{\T};\T)$}
    \psfrag{rate(h2)}{$\textrm{Ndof}^{-1}$}
    \includegraphics[trim={0 0 0 0},clip,width=4.5cm,height=4.5cm,scale=0.6]{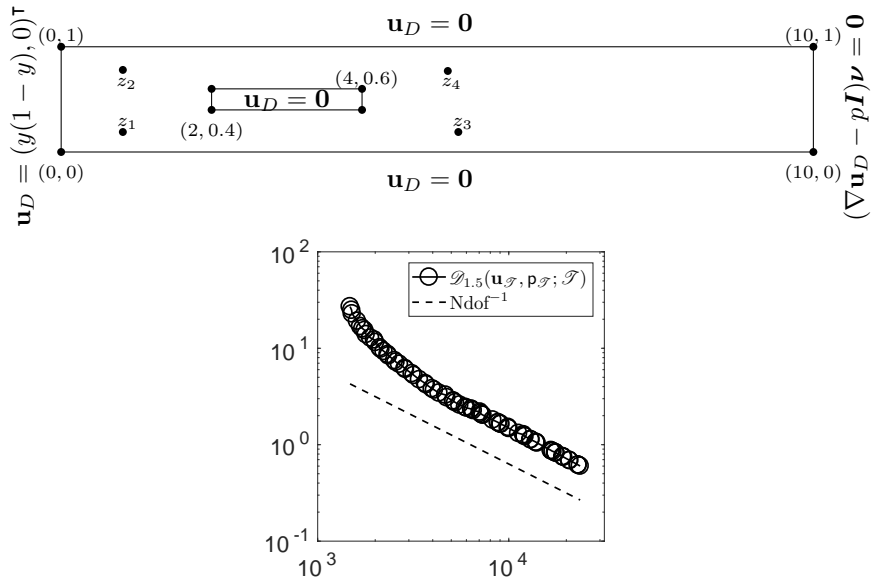}    
  \end{tabular}
\caption{\textbf{A rectangular domain with and obstacle:} Boundary conditions for problem \eqref{eq:NSEStrong} on $\Omega = (0,10) \times (0,1)\setminus[2,4]\times[0.4,0.6]$. We prescribe the parabolic Dirichlet inflow condition $\bu_{D}=(y(1-y),0)^\intercal$ on $\{ 0\} \times [0,1]$, $\bu_{D} = \boldsymbol{0}$ on $[0,10] \times \{ 0 \} \cup  \{ 1 \}$, $\bu_{D} = \boldsymbol{0}$ on the boundary of $[4,6] \times [0.4,0.6]$, and the homogeneous Neumann condition $(\nabla\bu_{D}-p\mathbf{I})\boldsymbol{\nu}=\boldsymbol{0}$ on $\{10\} \times [0,1] $ (top). Experimental rate of convergence for the error estimator $\mathscr{D}_{\alpha}(\bu_{\T},\pe_{\T};\T)$ with $\alpha=1.5$ (bottom).}
\label{fig:boundary_conditions_4points}
\end{figure}

\begin{figure}
\centering
\includegraphics[trim={0 0 0 0},clip,width=12.5cm,height=13cm,scale=0.8]{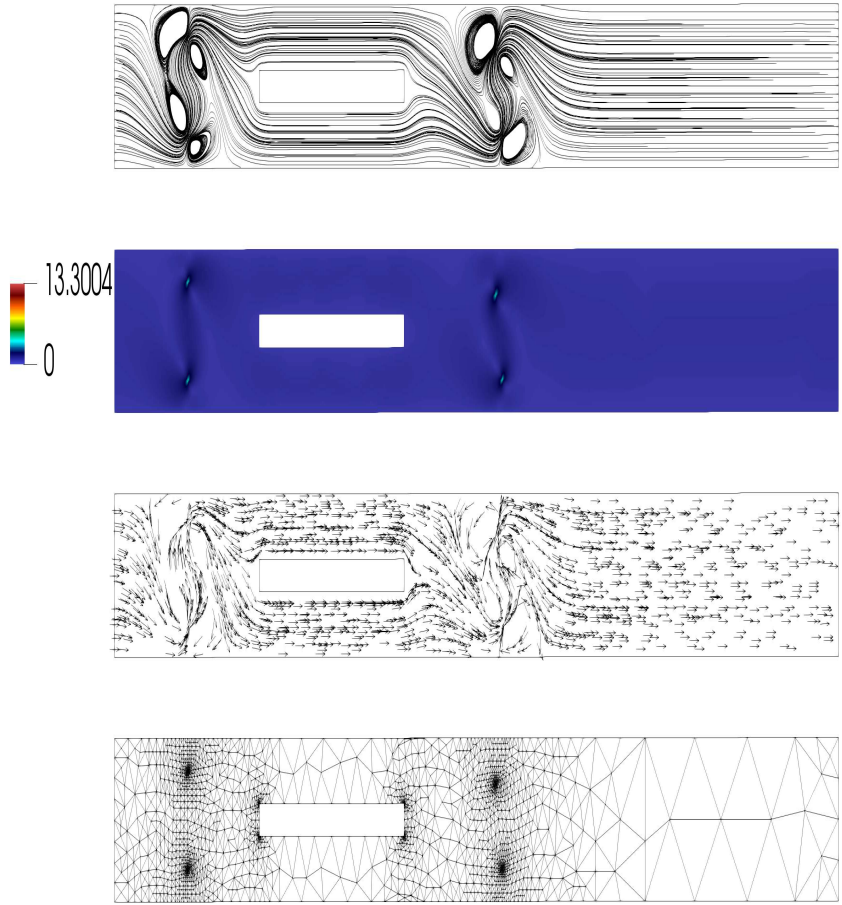}
\caption{\textbf{A rectangular domain with and obstacle.} Streamlines (first), magnitude of the flow field $\ue_\T$ (second), vector plot (third), and final mesh (last) after 100 adaptive refinements. The mesh consists of 5112 elements and 2679 vertices, where the adaptive cycles are obtained fixing $\alpha = 1.5$.}
\label{fig:final_obstacle}
\end{figure}

\bibliographystyle{siamplain}
\bibliography{biblio}
\end{document}